%% file: Eemeli_Blasten_licentiate_thesis_arXiv.tex
\documentclass[a4paper,12pt]{article}
\usepackage[latin1]{inputenc}
\usepackage[T1]{fontenc}
\usepackage[english]{babel}
\usepackage{amsmath, amsthm, amssymb}
\usepackage{mathrsfs}
\usepackage{pictex2}
\usepackage{multicol}
\usepackage{array}

      \theoremstyle{plain}
      \newtheorem{theorem}{Theorem}[subsection]
      \newtheorem*{theorem*}{Theorem}
      \newtheorem{lemma}[theorem]{Lemma}
      \newtheorem*{lemma*}{Lemma}
      \newtheorem{corollary}[theorem]{Corollary}
      \newtheorem*{corollary*}{Corollary}
      
      \newtheorem*{claim*}{Claim}
      \newtheorem{theoremSect}{Theorem}[section]
      \newtheorem{lemmaSect}[theoremSect]{Lemma}

      \theoremstyle{definition}
      \newtheorem{definition}[theorem]{Definition}
      \newtheorem*{definition*}{Definition}
      \newtheorem{definitionSect}[theoremSect]{Definition}
      
      \theoremstyle{remark}
      \newtheorem{remark}[theorem]{Remark}
      \newtheorem*{remark*}{Remark}
      
      \newtheorem*{example*}{Example}

\DeclareMathOperator{\supp}{supp}
\DeclareMathOperator*{\esssup}{ess\,sup}

\DeclareMathOperator{\pv}{p.v.}
\providecommand{\wbar}[1]{\overline{#1}}
\providecommand{\what}[1]{\widehat{#1}}
\providecommand{\F}{\mathscr{F}}
\providecommand{\Ca}{\mathscr{C}}
\providecommand{\C}{\mathbb{C}}
\providecommand{\R}{\mathbb{R}}
\providecommand{\N}{\mathbb{N}}
\providecommand{\Tr}{\operatorname{Tr}}
\renewcommand\Re{\operatorname{\mathfrak{Re}}}

\providecommand{\sgn}{\operatorname{sgn}}

\title{The inverse problem of the Schr\"odinger equation in the plane\\ A dissection of Bukhgeim's result}
\author{Eemeli Bl{\aa}sten}

\begin{document}
\begin{titlepage}
\centering
\vspace*{2cm}
\textsc{{\LARGE The inverse problem of the Schr\"odinger equation in the plane\\}\medskip{\Large A dissection of Bukhgeim's result}}

\vspace{2cm}
\textsc{\large Eemeli Bl{\aa}sten}

\vspace{1cm}
March 2011

\vfill
\begin{flushright}
\textsc{University of Helsinki}\\
Department of Mathematics and Statistics\\
Licentiate thesis\\
Advisor: Professor Lassi P\"aiv\"arinta
\end{flushright}
\end{titlepage}

\newpage
\thispagestyle{empty}
\section*{Forewords}
This is a revised version of my licentiate thesis. The original one was completed in June 2010. This version is partly rewritten and to be published in arXiv. These two versions differ a little, but only in the layout, comments and descriptions. Some mathematical typos were corrected and a cited preprint had appeared in a journal. All the theorems and proofs remain the same.

\vskip 2cm
\begin{flushright}
Eemeli Bl{\aa}sten,\\
March 2011
\end{flushright}

\newpage
\tableofcontents

\newpage
\section{Abstract}
The purpose of this licentiate thesis is to present Bukhgeim's result of 2007, which solves the inverse boundary value problem of the Schr\"odinger equation in the plane. The thesis is mainly based on Bukhgeim's paper \cite{bukhgeim} and Kari Astala's seminar talk, which he gave the 11\textsuperscript{th} and 18\textsuperscript{th} September of 2008 at the University of Helsinki.

Section \ref{sect_history} is devoted to the history and past results concerning some related problems: notably the inverse problem of the Schr\"odinger and conductivity equations in different settings. We also describe why some of the past methods do not work in the general case in a plane domain.

Section \ref{sect_summary} outlines Bukhgeim's result and sketches out the proof. This proof is a streamlined version of the one in \cite{bukhgeim} with the stationary phase method based on Kari Astala's presentation. In the following section we prove all the needed lemmas which are combined in section \ref{sect_proof} to prove the solvability of the inverse problem.

The idea of the proof is simple. Given two Schr\"odinger equations with the same boundary data we get an orthogonality relation for the solutions of the two equations. Then we show the existence of certain oscillating solutions and insert these into the orthogonality relation. Then by a stationary phase argument we see that the two Schr\"odinger equations are the same.

In the last section we contemplate an unclear detail in \cite{bukhgeim} which Kari Astala pointed out in his seminar talk: without an extra argument Bukhgeim's proof shows the solvability of the inverse problem only for differentiable potentials instead of ones in $L^p(\Omega)$. But the special oscillating solutions exist even for $L^p(\Omega)$ potentials.

\newpage
\section{Notation}

We use the following notation related to sets, functions and differential operators. We identify $\R^2$ and $\C$.
\begin{itemize}
\item $\Omega$ denotes the open unit disc centered at the origin of $\R^2$.

\item If $X\subset \R^2$ then $\partial X$ is its boundary.

\item $R = (z-z_0)^2 + (\wbar{z}-\wbar{z_0})^2$, where $z, z_0 \in \C$ are clear from the context.

\item $\partial_j$ denotes differentiation with respect to the variable $x_j$.

\item $\partial_n$ denotes the normal derivative on some boundary.

\item $\partial = \frac{1}{2}(\partial_1 - i \partial_2)$ and $\wbar{\partial}= \frac{1}{2}(\partial_1 + i \partial_2)$ denote the two complex derivatives.

\item $\Delta = \partial_1^2 + \partial_2^2 = 4\partial\wbar{\partial}$ is the Laplacian.
\end{itemize}

\begin{definition*}[Function spaces]
We will need the following few function spaces. The notation is standard except for the space of piecewise $W^{1,p}$ functions, which we denote by $W^{1,p}_\oplus$.

\begin{itemize}
\item $L^p(X)$, $1\leq p \leq \infty$, denotes the space of measurable functions $f$ such that $|f|^p$ is integrable when $p<\infty$ or such that $|f|$ is essentially bounded when $p=\infty$. Sometimes we write $L^p(X,z)$ to denote that the variable of $f$ is $z$. We use the norms
\begin{equation}
\begin{split}
\lVert f \rVert_{L^p(X)} &= \Big( \int_X |f(z)|^p \,dm(z) \Big)^{1/p},\quad p < \infty,\\
\lVert f \rVert_{L^\infty(X)} &= \esssup_{z\in X} |f(z)|.
\end{split}
\end{equation}

\item $C^\alpha(\wbar{X})$, $0 < \alpha < 1$, denotes the \emph{H\"older space} of continuous functions $f:\wbar{X} \to \C$ for which the H\"older norm
\begin{equation}
\lVert f \rVert_{C^\alpha(\wbar{X})} = \sup_{z_\in\wbar{X}} |f(z)| + \sup_{z_0\neq z_1} \frac{|f(z_0) - f(z_1)|}{|z_0-z_1|^\alpha}
\end{equation}
is finite. If $\alpha = 1$ we denote the space by $Lip(\wbar{X})$ instead of $C^1(\wbar{X})$. This is only used in one of the lemmas in the appendix.

\item $C^k(\wbar{X})$ with $k \in \N$ denotes the space of $k$ times continuously differentiable functions $\wbar{X}\to\C$. If $d$ is the dimention of the space, the norm is
\begin{equation}
\lVert f \rVert_{C^k(\wbar{X})} = \sum_{m_1+\cdots+m_d \leq k} \, \sup_{z\in\wbar{X}} \lvert \partial_1^{m_1} \cdots \partial_d^{m_d} f(z) \rvert.
\end{equation}

\item $W^{k,p}(X)$, $k\in\N$, $1\leq p \leq \infty$, denotes the \emph{Sobolev space} consisting of $L^p(X)$ functions $f$ whose distribution derivates up to order $k$ also belong to $L^p(X)$. The norm is defined by
\begin{equation}
\begin{split}
\lVert f \rVert_{W^{k,p}(X)} &= \Big( \sum_{|\alpha| \leq k} \big\lVert \partial^\alpha f \big\rVert_{L^p(X)}^p \Big)^{1/p}, p<\infty, \\
\lVert f \rVert_{W^{k,\infty}(X)} &= \max_{|\alpha| \leq k} \lVert \partial^\alpha f \rVert_{L^\infty(X)}.
\end{split}
\end{equation}

\item $W^{k,p}_\oplus(\Omega)$, $k \in \N$, $1\leq p \leq \infty$, denotes the space of piecewise $W^{k,p}$ functions. More precisely $W^{k,p}_\oplus(\Omega)$ is the set of those $f \in L^p(\Omega)$ for which there exist finitely many pairwise disjoint open sets $\Omega_j$ with Lipschitz boundary such that $f|_{\Omega_j} \in W^{k,p}(\Omega_j)$ and $\wbar{\Omega} = \cup \wbar{\Omega_j}$.
For the norm we define $\lVert f \rVert_{W^{k,p}_\oplus(\Omega)} = \big(\sum_j \lVert f_{|\Omega_j} \rVert_{W^{k,p}(\Omega_j)}^p \big)^{1/p}$ if $p<\infty$ and $\lVert f \rVert_{W^{k,\infty}_\oplus(\Omega)} = \max_j \lVert f_{|\Omega_j} \rVert_{W^{k,\infty}(\Omega_j)}$. Note that this definition of the norms does not depend on the choice of the sets $\Omega_j$.
\end{itemize}
\end{definition*}

\begin{definition*}
Given $s \in ]0,\infty]$ we denote by $\lVert \cdot \rVert_s$ either the $L^s(\Omega)$ or the $C^s(\wbar{\Omega})$ H\"older norm. The $L^s(\Omega)$ norm is chosen when $s \geq 1$ and the H\"older norm when $0<s<1$.
\end{definition*}

\newpage
\section{History}
\label{sect_history}
This small review of results concerning the inverse boundary value problems for the conductivity and Schr\"odinger equations is based on introductions in \cite{astalaPaivarinta} and \cite{nachman}. The inverse problem for the conductivity equation can be reduced to that of the Schr\"odinger equation. To transform the conductivity equation $\nabla \cdot (\gamma \nabla u) = 0$ into the Schr\"odinger equation $\Delta v + q v = 0$ it is enough to do the change of variables $u = \gamma^{-\frac{1}{2}} v$, $q = -\gamma^{-\frac{1}{2}} \Delta \gamma^{\frac{1}{2}}$. The Dirichlet-to-Neumann map for the new equation can be recovered from the boundary data of the old one: $\Lambda_q = \gamma^{-\frac{1}{2}}\big( \Lambda_\gamma + \frac{1}{2}\frac{\partial \gamma}{\partial n} \big) \gamma^{-\frac{1}{2}}$. The majority of the results cited below were proven for the conductivity equation or the Schr\"odinger equation having a potential of the conductivity type.

One of the early important papers on inverse boundary value problems is the famous paper of Calderón \cite{calderon}.
He considered an isotropic body $\Omega$ from which one would like to deduce the electrical conductivity $\gamma$ by doing electrical measurements on the boundary. If we keep the voltage $u$ fixed as $f$ on the boundary, then $u$ solves the boundary value problem
\begin{equation}
\begin{split}
\nabla \cdot (\gamma \nabla u) &= 0, \quad \Omega\\
u &= f, \quad \partial\Omega .
\end{split}
\end{equation}
The weighted normal derivative $\gamma \partial_n u$ is the current flux going out of $\Omega$. Calderón asked whether the knowledge of the Dirichlet-to-Neumann map $\Lambda_\gamma : f \mapsto \gamma \partial_n u_{|\partial\Omega}$ determines the conductivity $\gamma$ inside the whole domain $\Omega$. He was only able to show the injectivity of a linearized problem near $\gamma \equiv 1$.

Sylvester and Uhlmann solved the problem in dimensions $d\geq 3$ for smooth conductivities bounded away from zero \cite{sylvesterUhlmann}. They constructed solutions of the form $u_j = e^{x\cdot \,\zeta_j}\big( 1 + O(\frac{1}{|\zeta_j|}) \big)$, where the complex vectors $\zeta_j$ satisfy
\begin{equation}
\label{history_G_U_vectors}
\begin{split}
\zeta_1 &= i(k+m) + l,\\
\zeta_2 &= i(k - m) - l,
\end{split}
\end{equation}
where $l,k,m \in \R^d$ are perpendicular vectors satisfying $|l|^2 = |k|^2 + |m|^2$. Using a well-known orthogonality relation for the potentials $q_1$ and $q_2$ they got
\begin{equation}
0 = \int (q_1 - q_2) u_1 u_2 dx= \int (q_1 - q_2) e^{2i x \cdot k} \big( 1 + O(\frac{1}{|m|})\big) dx,
\end{equation}
and after taking $|m| \longrightarrow \infty$ they saw that the Fourier transforms of $q_1$ and $q_2$ are the same, so the potentials are too. Note that the only part that requires $d\geq 3$ in this solution is the existence of the three vectors $l,k,m$.

Some papers solve the Calderón problem in dimension two with various assumptions. Namely Kohn and Vogelius \cite{KVI} \cite{KVII}, Alessandrini \cite{alessandrini}, Nachman \cite{nachman} and finally Astala and P\"aiv\"arinta \cite{astalaPaivarinta}. Of these the first three require the conductivity to be piecewise analytic. Nachman required two derivatives to convert the conductivity equation into the Schr\"odringer equation. The paper of Astala and P\"aiv\"arinta solved Calderón's problem most generally: there were no requirements on the smoothness of the conductivity. It just had to be bounded away from zero and infinity, which is physically realistic.

There are also some results for the inverse boundary value problem of the Schr\"odinger equation whose potential is not assumed to be of the conductivity type. Jerison and Kenig proved in \cite{jerisonKenig} that if $q\in L^p(\Omega)$, $p>\frac{d}{2}$, $d\geq 3$, then the Dirichlet-to-Neumann map $\Lambda_q$ determines the potential $q$ uniquely. The case $d=2$ was open until the paper of Bukhgeim. He introduced in \cite{bukhgeim} new kinds of solutions to the Schr\"odinger equation, which allow the use of stationary phase. This led to the elegant solution of this long standing open problem. There is a point in the reasoning that is difficult to understand unless one assumes some differentiability for the potentials. In this thesis we try to clarify that point.

After \cite{bukhgeim} the new results in two dimensions have concerned the case of partial data i.e. whether or not the boundary data given only on a subset of the boundary determines the potential.
The two best results are from Imanuvilov, Uhlmann, Yamamoto \cite{imanuvilovUhlmannYamamoto} and Guillarmou and Tzou \cite{guillarmouTzou}. In the first paper the authors consider the Schr\"odinger equation in a plane domain and in the second one on a Riemann surface with boundary. The results of both papers state that knowing the Cauchy data on any open subset on the boundary determines the potential uniquely if it is smooth enough.

\newpage
\section{Summary}
\label{sect_summary}
In this licentiate thesis we present a streamlined proof of Bukhgeim's theorem, which is the uniqueness result for the inverse problem of the Schr\"odinger equation in a plane domain. The thesis is mainly based on two sources. The first one is Bukhgeim's paper \cite{bukhgeim} presenting his result. The second one is a talk given by Kari Astala in a seminar of functional analysis at the University of Helsinki the 11\textsuperscript{th} and 18\textsuperscript{th} of September in 2008. He presented Bukhgeim's proof and pointed out a problem in the density argument.

\subsection{The result}
Bukhgeim's theorem claims that if $q \in L^p(\Omega)$, $p>2$ then the boundary data $C_q$ for the Schr\"odinger equation $\Delta u + q u = 0$ in the unit disc on the plane determines $q$ uniquely. Because of an unclarity described in section \ref{mistakeAnalysis} we will only prove the theorem for piecewise $W^{1,p}$ potentials.
\begin{theorem*}
If $q_1, q_2\in W_\oplus^{1,p} (\Omega)$, $p>2$ and $C_{q_1} = C_{q_2}$, then $q_1 = q_2$.
\end{theorem*}

\subsection{Sketch of the proof}
The proof is divided into three independent steps. First we see that if $q_1$ and $q_2$ are two potentials giving the same boundary measurements then
\begin{equation}
\label{intro_ort}
\int_\Omega (q_1 - q_2)u_1 u_2 dm = 0
\end{equation}
for solutions $u_j$ of the Schr\"odinger equations $\Delta u_j + q_j u_j = 0$. The idea for this kind of orthogonality relation first appeared in Calderón's paper \cite{calderon}.

The next step is to show the existence of solutions $u_j$ which could give information about $q_1-q_2$ when plugged into the orthogonality relation \eqref{intro_ort}. One idea would be to create solutions such that $u_1u_2$ would be localized near some given point. Unfortunately because of maximum principles these kinds of solutions do not work for a large class of potentials $q_j$.

A more fruitful approach is to use the stationary phase method, i.e. to create solutions that oscillate wildly everywhere except near a given point $z_0 \in \Omega$. Thus the solution is zero in the mean everywhere except at $z_0$. These solutions with the relation \eqref{intro_ort} would tell us that $q_1(z_0) = q_2(z_0)$. The last step consists of proving this claim.
\medskip

The technical details are a bit more involved. Bukhgeim's ingenuity was to look for solutions of the form $u_1 = e^{in(z-z_0)^2}(1+r)$, $u_2 = e^{in(\wbar{z}-\wbar{z_0})^2}(1+s)$. To show the existence he used methods from \cite{vekua}. With the help of an auxiliary function $g$ he does the reduction
\begin{equation}
\Delta( e^{in(z-z_0)^2}f) + qe^{in(z-z_0)^2}f = 0 \Leftrightarrow \begin{cases}2\wbar{\partial}f &= e^{-inR} g\\ 2\partial g &= -qe^{inR}f \end{cases},
\end{equation}
where $R = (z-z_0)^2 + (\wbar{z}-\wbar{z_0})^2$. Then he transforms this into an integral equations by using the Cauchy operator $\Ca$ to get
\begin{equation}
\label{intro_int_eq}
f = 1 -\frac{1}{4}\Ca\big( e^{-inR} \wbar{\Ca}(e^{inR}q f)\big).
\end{equation}
To prove the existence of the solutions it remains to show that this operator is a contraction in a suitable Banach space when $n$ is large.

The last step of the whole proof is to show that the stationary phase method really works. After plugging the solutions gotten before into the orthogonality relation \eqref{intro_ort} we get
\begin{equation}
\int_\Omega e^{inR}(q_1 - q_2)(1 + r_n) dm = 0,
\end{equation}
where the H\"older norm of $r_n$ tends to zero. This step has two stages. First we show that $\int \frac{2n}{\pi} e^{inR}f dm$ tends to $f$ when $n$ grows. Then it remains to show that the remainder term $r_n$ causes no problems.

Bukhgeim showed that if $a$ is differentiable then
\begin{equation}
\label{intro_operator_id}
\int_\Omega \frac{2n}{\pi} e^{inR} a \,dm \longrightarrow a(z_0)
\end{equation}
and that the term with $r_n$ tends to zero. Unfortunately this does not imply the same for $a = q_1 - q_2 \in L^p(\Omega)$ even though $C^1(\wbar{\Omega})$ is dense in $L^p(\Omega)$. For a more in-depth analysis of this unclarity see section \ref{mistakeAnalysis}.

By assuming $q_1$ and $q_2$ to be piecewise in $W^{1,p}(\Omega)$, $p>2$ and it is possible to see that the stationary phase method works here. Basically it meant cleaning up Bukhgeim's paper of redundant argumentations and proving norm estimates for the stationary phase instead of using pointwise limits. Note that when $p>2$ the Sobolev space $W^{1,p}(\Omega)$ embeds into the H\"older space $C^{1-\frac{2}{p}}(\wbar{\Omega})$, so the case of singular potentials remains open.

\newpage
\section{Three steps}
This section contains the three main theorems used to prove the solvability of the inverse problem for the Schr\"odinger equation. First we show an $L^2$ orthogonality relation between solutions of the Schr\"odinger equations corresponding to potentials giving the same boundary data. The goal of the second subsection is to prove norm estimates guaranteeing that the stationary phase method works. The last subsection is devoted to proving the existence of suitably oscillating solutions to the Schr\"odinger equation.

\subsection{Orthogonality}
\label{orthogonality}

In this part we will define the set of boundary data corresponding to a potential $q$ and show that there is an orthogonality relation between solutions of Schr\"odinger equations having the same boundary data. The approach is well known now. According to an overview on inverse boundary value problems written by Sylvester \cite{sylvester} a variation of this relation first appeared in \cite{calderon} and its modern form appeared in \cite{alessandrini}.
\medskip

First we define the boundary values and normal derivative of a $W^{2,p}(\Omega)$ function. This is needed only to define the boundary data of the Schr\"odinger equation, so we may use quite sloppy norm estimates. The continuity of the trace mapping $\Tr: u \mapsto u_{|\partial\Omega}$ is proved in \cite[ch. 5.5, thm 1]{evans}. Because $\Omega$ is the unit disc we may define the normal derivative $\partial_n u$ explicitly.

\begin{lemma}
\label{dnCont}
The differential operator $\partial_n : W^{2,p}(\Omega) \to L^p(\partial\Omega)$ mapping $u$ to its normal derivative on the boundary $\partial_n u(z) = \sum_{j=1}^2 z_j \Tr \partial_j u(z)$ is bounded when $1 \leq p \leq \infty$.
\end{lemma}
\begin{proof}
The trace mapping $\Tr : W^{1,p}(\Omega) \to L^p(\partial\Omega)$ is bounded \cite[ch. 5.5, thm 1]{evans} and so are the partial derivatives $\partial_j : W^{2,p}(\Omega) \to W^{1,p}(\Omega)$. The functions $z\mapsto z_j$ from $\partial\Omega$ to $\R$ are bounded, so the multiplication operator $m_j : u\mapsto z_j u$ is bounded in $L^p(\partial\Omega)$. Thus $\partial_n$ is bounded as a sum of the bounded maps $m_j \circ \Tr \circ \,\partial_j : W^{2,p}(\Omega) \to L^p(\partial\Omega)$.
\end{proof}

Now we can define the boundary data for the Schr{\"o}dinger equation. By defining the set of boundary data instead of the Dirichlet-to-Neumann map we avoid problems arising from the possible nonuniqueness of the Cauchy problem $\Delta u + q u = 0$, $\Tr u = f$.
\begin{definition}
Given a measurable function $q:\Omega \to \C$, we define by
\begin{equation}
C_q = \big\{ (\Tr u, \partial_n u) \mid u \in W^{2,p}(\Omega), \Delta u + q u = 0 \big\}
\end{equation}
the \emph{Cauchy data} or \emph{boundary data} given by the potential $q$.
\end{definition}
\medskip

Our next task is to prove the orthogonality relation. The idea is to first transform $\int_\Omega (q_1 - q_2) u_1u_2$ into an integral over the boundary $\partial\Omega$. Then we show an integration by parts formula on the boundary. To finish we use the fact that the boundary data for both Schr\"odinger equations are the same. This will cancel out the boundary integral.

\begin{lemma}
\label{functionalToBoundary}
Let $q_1$ and $q_2$ be measurable functions $\Omega \to \C$. If $u_1, u_2 \in W^{2,p}(\Omega)$, $2 \leq p \leq \infty$ satisfy the equations
\begin{equation}
\left\{ \begin{array}{c}
	\Delta u_1 + q_1 u_1 = 0\\
	\Delta u_2 + q_2 u_2 = 0
\end{array} \right. ,
\end{equation}
then
\begin{equation}
\int_\Omega (q_1 - q_2) u_1 u_2 \,dm(z)  = \int_{\partial\Omega} u_1 \partial_n u_2 - u_2 \partial_n u_1 \,d\sigma(z).
\end{equation}
\end{lemma}
\begin{proof}
Note that $(q_1 - q_2) u_1 u_2 = u_1 \Delta u_2 - u_2 \Delta u_1$ and $\Delta = \partial_1^2 + \partial_2^2$. Because $u_1, u_2 \in W^{2,p}(\Omega)$, $p\geq 2$ we have $u_1, u_2, \partial_k u_1, \partial_k u_2 \in W^{1,2}(\Omega)$. Using integration by parts (lemma \ref{soboIntByParts}) to the function pairs $(u_1, \partial_1 u_2)$, $(u_1, \partial_2 u_2)$, $(\partial_1 u_1, u_2)$ and $(\partial_2 u_1, u_2)$ we get
\begin{equation}
\begin{split}
&\int_\Omega (q_1 - q_2)u_1 u_2 dm = \int_\Omega (u_1 \Delta u_2 - u_2 \Delta u_1) dm \\
&= \sum_{k=1}^2 \int_{\partial\Omega} n_k (u_1 \partial_k u_2 - u_2 \partial_k u_1) d\sigma - \sum_{k=1}^2 \int_\Omega (\partial_k u_1 \partial_k u_2 - \partial_k u_2 \partial_k u_1) dm \\
&= \int_{\partial\Omega} (u_1 n\cdot \nabla u_2 - u_2 n\cdot \nabla u_1) d\sigma = \int_{\partial\Omega} (u_1 \partial_n u_2 - u_2 \partial_n u_1) d\sigma.
\end{split}
\end{equation}
\end{proof}

\begin{corollary}
\label{switchOnBoundary}
If $q:\Omega \to \C$ is measurable and the functions $u,v \in W^{2,p}(\Omega)$, $2 \leq p \leq \infty$, satisfy $\Delta u + q u = 0$ and $\Delta v + q v = 0$, then
\begin{equation}
\int_{\partial\Omega} u \partial_n v \,d\sigma(z) = \int_{\partial\Omega} v \partial_n u \,d\sigma(z).
\end{equation}
\end{corollary}
\begin{proof}
Choose $q_1 = q_2 = q$ in the previous lemma.
\end{proof}

\begin{theorem}
\label{ORTthm}
Let $q_1, q_2:\Omega \to \C$ be two measurable potentials such that $C_{q_1} = C_{q_2}$.
Then if $u_1, u_2 \in W^{2,p}(\Omega)$, $2\leq p \leq \infty$, satisfy $\Delta u_1 + q_1 u_1 = 0$ and $\Delta u_2 + q_2 u_2 = 0$, they also satisfy
\begin{equation}
\label{ORTrel}
\int_{\Omega} (q_1 - q_2) u_1 u_2 \,dm(z) = 0.
\end{equation}
\end{theorem}
\begin{proof}
By lemma \ref{functionalToBoundary} $\int_\Omega (q_1 - q_2) u_1 u_2 \,dm(z) = \int_{\partial\Omega} u_1 \partial_n u_2 - u_2 \partial_n u_1 \,d\sigma(z)$. Now let $U_2 \in W^{2,p}(\Omega)$ be a solution to the Cauchy problem
\begin{equation}
\begin{cases}
\Delta U_2 + q_2 U_2 = 0\\
U_{2|\partial\Omega} = u_{1|\partial\Omega}\\
\partial_n U_2 = \partial_n u_1
\end{cases}
\end{equation}
We know that $(u_{1|\partial\Omega}, \partial_n u_1) \in C_{q_1}$ so it is also in $C_{q_2}$. Thus $U_2$ exists.

By $U_{2|\partial\Omega} = u_{1|\partial\Omega}$ and corollary \ref{switchOnBoundary} we have
\begin{equation}
\int_{\partial\Omega} u_1 \partial_n u_2 \,d\sigma(z) = \int_{\partial\Omega} U_2 \partial_n u_2 \,d\sigma(z) = \int_{\partial\Omega} u_2 \partial_n U_2\,d\sigma(z).
\end{equation}
All in all
\begin{equation}
\int_{\Omega} (q_1 - q_2) u_1 u_2 \,dm(z) = \int_{\partial\Omega} u_2\partial_n U_2 - u_2 \partial_n u_1\,d\sigma(z) = 0,
\end{equation}
because $U_2$ was chosen such that $\partial_n U_2 = \partial_n u_1$.
\end{proof}

\newpage
\subsection{Stationary phase method}
\label{approx}
The purpose of this part is to show that the stationary phase method works in this setting. In the previous section we proved the orthogonality relation $\int(q_1 - q_2)u_1 u_2 = 0$. In section \ref{exist} we will show the existence of solutions of the form $u_1 = e^{in(z-z_0)^2}(1+r)$, $u_2 = e^{in(\wbar{z}-\wbar{z_0})^2}(1+s)$ with the remainder terms $r$ and $s$ depending on $z_0$ and $n$. Thus in this section we need to study oscillatory integral operators of the form $f \mapsto \int e^{inR(z,z_0)}f(z)g_n(z,z_0) dm(z)$, with $R = (z-z_0)^2 + (\wbar{z}-\wbar{z_0})^2$.

The main idea is to note that $e^{inR}$ is a Gaussian kernel. The Fourier transform of a Gaussian is also a Gaussian and the convolution operator transforms into multiplication. The $L^2$ theory of Fourier multipliers is trivial because of Parseval's theorem. Thus Fourier analysis describes completely the behaviour of the main term $\int e^{inR}(q_1-q_2)dm$.

Problems arise from the remainder term. The remaining part of the oscillating solution decays like $n^{-\beta}$, $\beta \in ]0,1[$. Unfortunately the main part has to be multiplied by $n$ to give information about the potentials. This prevents the use of the triangle inequality naively because the remaining integral would grow to infinity. Also, the remainder depends on $z_0$ and $n$ so it could cancel out oscillations of the kernel $e^{inR}$. To solve these issues we need to assume some smoothness for the potentials. Then we may integrate by parts to get rid of the factor $n$.

Handling the remainder term is the only place where the potentials need to have some smoothness. If one would like to solve the inverse problem more generally using methods based on Bukhgeim's proof, one only needs to prove a better estimate instead of theorem \ref{OPremainder}.
\medskip

In the first part of this section we prove that the Fourier transform of a two-dimensional complex Gaussian kernel is also a complex Gaussian kernel. To do this we need the Fourier transform of a one-dimensional Gaussian and the Cauchy integral theorem. We assume that the reader knows some basic facts about analytic functions and Fourier transforms of tempered distributions. We choose the following form for the transform:

\begin{definition}
The Fourier transform and its inverse $\F, \F^{-1}$, mapping $L^2(\R^d)$ to itself, are defined by the formulae
\begin{equation}
\begin{split}
\F f (\xi) &= \widehat{f}(\xi) = \frac{1}{(2\pi)^{d/2}} \int_{\R^d} e^{-i(\xi_1 z_1 + \cdots +\xi_d z_d)} f(z) dm(z),\\
\F^{-1}g (z) &= \check{g}(z) = \frac{1}{(2\pi)^{d/2}} \int_{\R^d} e^{i(z_1 \xi_1 + \cdots + z_d \xi_d)} g(\xi) dm(\xi),
\end{split}
\end{equation}
where $d$ is the dimension of the space. These are well defined. For proofs and other properties we will use see \cite[Section 4.3.1. thms 1, 2]{evans}.
\end{definition}

\input{input_gaussianFourier.tex}

\begin{lemma}
\label{convKernTransf}
For each nonzero $n\in\R$ the function $\kappa_n:\C \to \C$, given by $\kappa_n(z) = \frac{2n}{\pi} e^{i n (z^2 + \wbar{z}^2)}$ is a tempered distribution and
\begin{equation}
	\widehat{\kappa_n}(\xi) = \frac{\sgn(n)}{2\pi}e^{-i\frac{\xi^2 + \wbar{\xi}^2}{16n}}.
\end{equation}
\end{lemma}
\begin{proof}
The function $\kappa_n$ is bounded and measurable so it is a tempered distribution. Let $\varphi \in \mathscr{S} (\R)$ be a Schwartz test function. Note that by the previous lemma (\ref{gaussianFourier}) if $c>0$ then
\begin{equation}
\int_\R e^{-c \frac{t^2}{2}} \what{\varphi}(t) dt = \int_\R \frac{1}{\sqrt{c}} e^{-\frac{\xi^2}{2c}} \varphi(\xi) d\xi.
\end{equation}
Let us choose a branch of the square root in the complex plane such that $\operatorname{arg} \sqrt{z} \in \left] -\frac{\pi}{2}, \frac{\pi}{2} \right]$. Now both sides of the previous equation are analytic functions of $c$ in the right half-plane $\Re (c) > 0$. Thus they are also equal in the right half-plane. 

Let $\phi, \psi \in \mathscr{S}(\R)$ be two Schwartz test functions. By Fubini's theorem and dominated convergence
\begin{equation}
\begin{split}
\int_{\C} & e^{in(z^2 + \wbar{z}^2)} \big(\phi(\xi_1)\psi(\xi_2)\big)^\wedge(z) dm(z) = \int_{\R^2} e^{i2n(x^2 - y^2)} \what{\phi}(x)\what{\psi}(y) dm(x,y) \\
&= \int_{-\infty}^\infty e^{i 2n x^2} \what{\phi}(x) dx \int_{-\infty}^\infty e^{-i 2n y^2} \what{\psi}(y) dy \\
&= \lim_{\epsilon \to 0+} \int_{-\infty}^\infty e^{-(2\epsilon - i4n) \frac{x^2}{2}} \what{\phi}(x) dx \int_{-\infty}^\infty e^{-(2\epsilon + i4n) \frac{y^2}{2}} \what{\psi}(y) dy \\
&= \lim_{\epsilon\to 0+} \frac{1}{\sqrt{2\epsilon - i 4n}}\int_{-\infty}^\infty \!\!\!\!\!\! e^{-\frac{\xi_1^2}{4\epsilon - i8n}} \phi(\xi_1) d\xi_1 \frac{1}{\sqrt{2\epsilon + i 4n}} \int_{-\infty}^\infty \!\!\!\!\!\! e^{-\frac{\xi_2^2}{4\epsilon + i8n}} \psi(\xi_2) d\xi_2 \\
&= \frac{1}{\sqrt{\lvert 4n \rvert}}e^{i\frac{\pi}{4} \sgn(n)} \frac{1}{\sqrt{\lvert 4n \rvert}} e^{-i\frac{\pi}{4} \sgn(n)} \int_{-\infty}^\infty e^{\frac{\xi_1^2 - \xi_2^2}{i8n}}\phi(\xi_1)\psi(\xi_2) dm(\xi_1,\xi_2) \\
&= \frac{1}{4|n|} \int_\C e^{-i \frac{\xi^2 + \wbar{\xi}^2}{16n}} \phi(\xi_1) \psi(\xi_2) dm(\xi),
\end{split}
\end{equation}
so $\mathscr{F}\big\{ \frac{2n}{\pi} e^{in(z^2 + \wbar{z}^2)} \big\}(\xi) = \frac{\sgn(n)}{2\pi} e^{-i\frac{\xi^2 + \wbar{\xi}^2}{16n}}$.
\end{proof}
\medskip

Now we have all the ingredients to study the oscillatory integral operator $f \mapsto \int n e^{inR} f dm$. When applied to $(q_1 - q_2) (1+r)(1+s)$ we must use different methods to the principal term $(q_1 - q_2)$ and the remainder terms $(q_1 - q_2)(r+s+rs)$ because the latter depends on both variables of the kernel $e^{inR}$.

\begin{theorem}
\label{OPdefThm}
Let $n>0$. Then the operator defined by
\begin{equation}
T_n f (z_0) = \int_\Omega \frac{2n}{\pi} e^{in\big((z-z_0)^2 + (\wbar{z}-\wbar{z_0})^2\big)} f(z) \,dm(z)
\end{equation}
maps $L^2(\Omega)$ to $L^2(\C)$ isometrically and its Fourier transform is given by $\F\{T_n f\}(\xi) = e^{-i\frac{\xi^2 + \wbar{\xi}^2}{16n}}\widehat{f}(\xi)$, where $f$ is extended as zero outside of $\Omega$.
\end{theorem}
\begin{proof}
We expand $f$ to $\C$ by choosing $f = 0$ in $\C\setminus\Omega$. Now we can interpret $T_n f$ as the convolution $\kappa_n \ast f$. Because $f$ is compactly supported the convolution gives a tempered distribution whose Fourier transform is $\F\{T_n f\}(\xi) = 2\pi \widehat{\kappa_n} \widehat{f}(\xi) = e^{-i\frac{\xi^2 + \wbar{\xi}^2}{16n}} \widehat{f}(\xi)$ by lemma \ref{convKernTransf}.

By Parseval's theorem and the fact that $\xi^2 + \wbar{\xi}^2 \in \R$ we get
\begin{equation}
\begin{split}
\big\lVert T_n f \big\rVert_{L^2(\C)} &= \big\lVert \F\{T_n f \} \big\rVert_{L^2(\C)} = \big\lVert e^{-i\frac{\xi^2 + \wbar{\xi}^2}{16n}} \widehat{f}(\xi) \big\rVert_{L^2(\C)} = \big\lVert \widehat{f}(\xi) \big\rVert_{L^2(\C)} \\
&= \lVert f \rVert_{L^2(\C)} = \lVert f \rVert_{L^2(\Omega)},
\end{split}
\end{equation}
because $\supp f \subset \wbar{\Omega}$.
\end{proof}

The next theorem shows that the stationary phase method works for $L^2$ potentials. The argumentation is similar to the one in the previous theorem. These two theorems describe completely the behaviour of the principal term of the oscillatory integral.

\begin{theorem}
\label{OPthm}
Let $f \in L^p(\Omega)$, $p \geq 2$. Then 
\begin{equation}
T_n f (z_0) = \int_{\Omega} \frac{2n}{\pi} e^{i n \big((z-z_0)^2 + (\wbar{z}-\wbar{z_0})^2\big)} f(z)\,dm(z) \longrightarrow f(z_0)
\end{equation}
in $L^2(\Omega)$ when $n\longrightarrow \infty$.
\end{theorem}
\begin{proof}
Because $\Omega$ is bounded $f$ is in $L^2(\Omega)$. The previous theorem gives $T_n f_{|\Omega} \in L^2(\Omega)$. By continuing $f$ as zero outside $\Omega$ we may interpret $T_nf$ as the convolution $\kappa_n \ast f$. By Parseval's theorem and dominated convergence
\begin{equation}
\begin{split}
	\big\lVert f - &T_n f \big\rVert_{L^2(\Omega)}^2 = \big\lVert f - \kappa_n \ast f \big\rVert_{L^2(\Omega)}^2 \leq \big\lVert f - \kappa_n \ast f \big\rVert_{L^2(\C)}^2 \\
	&= \big\lVert \widehat{f} - e^{-i\frac{\xi^2 + \wbar{\xi}^2}{16n}} \widehat{f} \big\rVert_{L^2(\C)}^2 = \int_\C \big| 1- e^{-i\frac{\xi^2 + \wbar{\xi}^2}{16n}} \big|^2 |\widehat{f}(\xi)|^2 \,dm(\xi) \longrightarrow 0
\end{split}
\end{equation}
when $n \longrightarrow \infty$.
\end{proof}

\medskip

The last theorem of this section is the most technical. In it we prove that the remaining terms of the oscillatory integral cause no problems. Note that this theorem is the only place in this thesis where we must assume some smoothness on the potentials $q_j$ of the Schr\"odinger equations $\Delta u + q_j u = 0$. The reason to require smoothness is to be able to integrate by parts to get rid of the growing factor $n$.

The proof contains some of the ideas of Bukhgeim's lemma 3.4, which is the most difficult one in his paper \cite{bukhgeim}. The rest is in lemma \ref{restNorms} where we show that the Sobolev norms of the remainder terms tend to zero.

\input{input_OPremainder.tex}

\newpage
\subsection{Existence of suitable solutions}
\label{exist}
Next we will show that the Schr\"odinger equation has special kinds of oscillating solutions. The main point is to find solutions which, when plugged into the orthogonality relation \eqref{ORTrel}, give the convolution kernel $\kappa_n$, which was studied in lemma \ref{convKernTransf} and theorem \ref{OPdefThm}, and some error terms.

To find the special solutions we factor the Laplacian and transform the Schr\"odinger equation into an integral equation. In the plane the Laplacian factors into $\Delta = 4\partial\wbar{\partial}$. This gives a hint that the kernel $K_n \approx u_1 u_2$ of the integral operator $T_n$ should also be factored into a holomorphic and an antiholomorphic\footnote{A function $u$ is holomorphic if $\wbar{\partial} u =0$ and antiholomorphic if $\partial u = 0$.} part. Thus following Bukhgeim's idea we seek solutions of the forms $u_1 \approx e^{in(z-z_0)^2}$ and $u_2 \approx e^{in(\wbar{z}-\wbar{z_0})^2}$.

The integral equation will have the form of a fixed point equation. To prove that it has a solution we will show that for big $n$ the operator is a contraction in the H\"older space $C^\alpha(\wbar{\Omega})$, $\alpha = 1 -\frac{2}{p}$. To show this we will integrate by parts inside a Cauchy operator. Its kernel has a singularity which we must remove using a suitable cut-off function. We start by constructing it and calculating some norm bounds.

\label{solutions}
\begin{lemma}
\label{perusSilea}
Define $\gamma:\R \to \R$ as
\begin{equation}
\gamma(x) = \frac{g(2-x)}{g(2-x) + g(x-1)}, \quad \text{where} \quad g(x) = \left\{ \begin{array}{rl} 0,& \text{when $x \leq 0$}\\ e^{-\frac{1}{x}},& \text{when $x>0$} \end{array} \right. .
\end{equation}
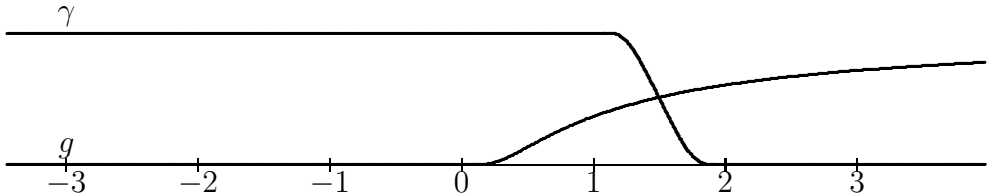
\begin{centering} \mbox{\input{basetestfunct1image.tex}} \end{centering}

Then $g$ and $\gamma$ are in $C^\infty(\R)$. Moreover 

\begin{enumerate}
\item $0 \leq \gamma \leq 1$,
\item $\gamma(x) = 1 \Leftrightarrow x \leq 1$,
\item $\gamma(x) = 0 \Leftrightarrow x \geq 2$,
\item $\lVert \gamma \rVert_{L^\infty(\R)} = 1$ and $\lVert \gamma' \rVert_{L^\infty(\R)} \leq 8 e$.
\end{enumerate}
\end{lemma}
\begin{proof}
We take it as a well known fact that $g$ is smooth. Thus $\gamma$ is smooth if the denominator is never zero. Let us find a lower bound for $g(2-x)+g(x-1)$.

The function $g$ is always non-negative. Thus $g(2-x) + g(x-1) \geq \max\{g(2-x), g(x-1)\}$. Because $g$ is increasing, $g(2-x)$ is decreasing and $g(x-1)$ increasing. We have $g(2-x) \geq g(x-1) \Leftrightarrow x \leq \frac{3}{2}$ and thus $\max\{g(2-x), g(x-1)\} \geq g(2-\frac{3}{2}) = g(\frac{1}{2}) = e^{-2}$. This implies the smoothness of $\gamma$.

Because $g(x) \geq 0$ for all $x\in\R$, we have either $\gamma(x) = 0$, or $\gamma(x) = \big(1 + g(x-1)/g(2-x)\big)^{-1} \leq 1$, so property 1 holds. Also $\gamma(x) = 1$ if and only if $g(x-1) = 0$ which means that $x \leq 1$. Moreover $\gamma(x) = 0$ if and only if $g(2-x) = 0$ so $x \geq 2$. Thus properties 2 and 3 hold.

Properties 1 and 2 imply $\lVert \gamma \rVert_{L^\infty(\R)} = 1$. Clearly $\gamma'(x) = 0$ if $x \not\in \left]1,2\right[$. If $1 < x < 2$, then
\begin{equation}
\begin{split}
|\gamma'(x)| &= \Big| \frac{-g'(2-x)g(x-1) - g(2-x)g'(x-1)}{\big(g(2-x)+g(x-1)\big)^2} \Big| \\
&\leq \frac{2\lVert g \rVert_{L^\infty(]0,1[)} \lVert g' \rVert_{L\infty(]0,1[)}}{\min\big(g(2-x)+g(x-1)\big)^2} = \frac{2\cdot e^{-1} \cdot 4 e^{-2}}{e^{-4}} = 8e.
\end{split}\end{equation}
\end{proof}

Next we build the cut-off function. The integral operator whose fixed point we wish to find will have weights of the form $(z-z_0)^{-1}$ after integration by parts. Thus we will also need some norm estimates for $h(z-z_0)^{-1}$ and its derivative, where $h$ is the cut-off function. We will use a mean value inequality of the form ``if $\Omega$ is convex and $f \in C^1(\Omega,\C)$ then $|f(x)-f(y)| \leq \big\lVert |\nabla f | \big\rVert_\infty |x-y|$'', which is easily reduced to the case of a real function on an interval.

\begin{lemma}
\label{testFexist}
Let $z_0 \in \Omega$ and $1 > \delta > 0$. Then there exists a test function $h \in C^\infty_0(\Omega)$ which has the properties
\begin{enumerate}
\item $0 \leq h \leq 1$,
\item $h(z) = 0$ when $|z-z_0| < \frac{\delta}{2}$,
\item $h(z) = 1$ when $\delta \leq |z-z_0|$ and $|z| \leq 1-\delta$,
\item If $\mho := \{z \in \Omega \mid h(z) \neq 1\}$, then $m(\mho) \leq 2\pi\delta$.
\item $\lVert \frac{h}{z-z_0} \rVert_\infty, \lVert \frac{h}{\wbar{z}-\wbar{z_0}} \rVert_\infty \leq \frac{2}{\delta}$, $\lVert \frac{h}{z-z_0} \rVert_{C^1(\Omega)}, \lVert \frac{h}{\wbar{z}-\wbar{z_0}} \rVert_{C^1(\Omega)} \leq \frac{200}{\delta^2}$,
\end{enumerate}
\end{lemma}
\begin{proof}
Take $\gamma$ as in the previous lemma.

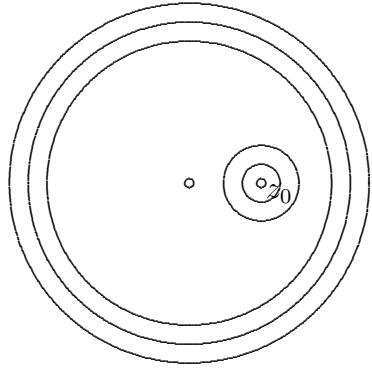
\begin{multicols}{2}
\begin{tabular}{c}
	\mbox{\input{testfunct3imageuusi.tex}}\\
	\mbox{\input{testfunct2imageuusi.tex}}
\end{tabular}

Define $\gamma_S(z) = 1 -\gamma\big(\frac{2}{\delta}|z-z_0|\big)$. Now $0 \leq \gamma_S \leq 1$, $\gamma_S(z) = 0$ when $|z-z_0| < \frac{\delta}{2}$ and $\gamma_S(z) = 1$ when $|z-z_0| \geq \delta$.

To get $h$ compactly supported take $\gamma_B(z) = \gamma\Big(\frac{2}{\delta}\big(|z|-(1-\delta)\big)+1\Big)$. Now $0 \leq \gamma_B \leq 1$, $\gamma_B(z) = 0$ when $|z| \geq 1 -\frac{\delta}{2}$ and $\gamma_B(z) = 1$ when $|z|\leq 1 - \delta$.

Finally we define
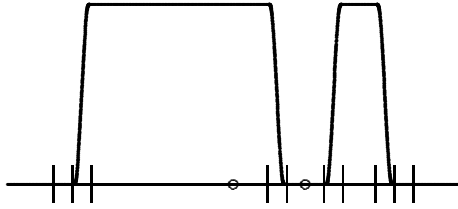
\begin{equation}
h(z) = \gamma_S(z) \gamma_B(z).
\end{equation}
Properties 1, 2 and 3 are now satisfied and $h\in C^\infty_0(\Omega)$. Moreover $m(\mho) \leq \pi - \pi(1-\delta)^2 + \pi \delta^2 = 2\pi \delta$.
\end{multicols}

To prove the last claim we need to calculate the norms of $h$. This is done using the two-dimensional chain rule and the estimates of the previous lemma for $\gamma$.
\begin{equation}
\begin{split}
\lVert h \rVert_\infty &\leq \lVert \gamma_S \rVert_\infty \lVert \gamma_B \rVert_\infty = 1 \\
| \nabla h | &= \big| \gamma_B \nabla \gamma_S + \gamma_S \nabla \gamma_B \big| \leq \big| \nabla \gamma_S \big| + \big| \nabla \gamma_B \big| \\
&= \Big| \nabla \gamma\big( \frac{2}{\delta}|z-z_0|\big) \Big| + \Big| \nabla \gamma\Big(\frac{2}{\delta}\big(|z|-(1-\delta)\big)+1\Big) \Big| \\
&= \Big| \frac{2}{\delta} \frac{(x-x_0,y-y_0)}{|z-z_0|} \gamma'\Big(\frac{2}{\delta}|z-z_0|\Big)\Big| \\
&\phantom{=}+ \Big| \frac{2}{\delta} \frac{(x,y)}{|z|} \gamma'\Big(\frac{2}{\delta}\big(|z|-(1-\delta)\big)+1\Big)\Big|\\
&\leq \frac{4}{\delta} \lVert \gamma' \rVert_\infty \leq \frac{32e}{\delta} \leq \frac{96}{\delta}.
\end{split}
\end{equation}
To calculate the Lipschitz norms of $h(z-z_0)^{-1}$ we will use the mean value inequality. Note that $h(z) = 0$ when $|z-z_0| < \delta / 2$. This gives us the estimates
\begin{equation}
\begin{split}
\Big\lVert & \frac{h}{z-z_0} \Big\rVert_\infty, \Big\lVert \frac{h}{\wbar{z}-\wbar{z_0}} \Big\rVert_\infty \leq \Big\lVert \frac{1}{|z-z_0|} \Big\rVert_{L^\infty(\Omega\setminus B(z_0, \delta/2))} = \frac{2}{\delta}, \\
\Big\lVert & \frac{h}{z-z_0} \Big\rVert_{C^1(\wbar{\Omega})} = \Big\lVert \frac{h}{z-z_0} \Big\rVert_\infty + \sup_{z_1, z_2 \in \Omega} \frac{\Big| \frac{h(z_1)}{z_1-z_0} - \frac{h(z_2)}{z_2-z_0} \Big|}{|z_1-z_2|} \\
&\leq \frac{2}{\delta} + \sup_{\substack{z\in\Omega\\ |z-z_0| > \frac{\delta}{2}}} \Big| \nabla \frac{h(z)}{z-z_0} \Big| \leq \frac{2}{\delta}\\
&\phantom{\leq} + \lVert h \rVert_\infty \sup_{\substack{z\in\Omega\\ |z-z_0| > \frac{\delta}{2}}} \Big| \Big( \frac{-1}{(z-z_0)^2}, \frac{-i}{(z-z_0)^2} \Big) \Big| + \sup_{\substack{z\in\Omega\\ |z-z_0| > \frac{\delta}{2}}} \Big| \frac{1}{z-z_0} \Big| |\nabla h(z)| \\
&\leq \frac{2}{\delta} + \frac{4\sqrt{2}}{\delta^2} + \frac{2}{\delta}\cdot \frac{96}{\delta} \leq \frac{200}{\delta^2}.
\end{split}
\end{equation}
The estimate for the Lipschitz norm of $h(\wbar{z}-\wbar{z_0})^{-1}$ is deduced similarly.
\end{proof}

\medskip

How to reduce the Schr\"odinger equation $\Delta u + q u = 0$ into an integral equation? In fact we want to solve for the remainder terms $f$ and $g$ in the equations $\Delta u_1 + q_1 u_1 = 0$, $u_1 = e^{in(z-z_0)^2}(1+f)$ and $\Delta u_2 + q_2 u_2 = 0$, $u_2 = e^{in(\wbar{z}-\wbar{z_0})^2}(1+g)$. We will consider the first equation with a generic potential $q$ and later on, starting in theorem \ref{operatorNorm}, we take $q_1$, $q_2$ and the complex conjugate in the exponential into account.

Write $\psi = (z-z_0)^2$, $u = e^{in\psi}w$ and plug it into $\Delta u + q u = 0$. In the plane we have the formula $\Delta = 4\partial\wbar{\partial}$, so we get $4\partial \big(e^{in\psi} \wbar{\partial} w \big) + q e^{in\psi} w=0$. If we denote $V = e^{in\psi}\wbar{\partial} w$ we get a system of first order equations
\begin{equation}
\begin{cases}
\wbar{\partial} w &= e^{-in\psi}V\\
\partial V &= -\frac{1}{4}qe^{in\psi}w
\end{cases}.
\end{equation}
The problem here is that $|e^{\pm in\psi}|$ grows too fast when $n\longrightarrow\infty$. This can be solved by choosing $V = e^{-in\wbar{\psi}}v$ to get
\begin{equation}
\label{schrodingerReduced}
\begin{cases}
\wbar{\partial} w &= e^{-inR}v\\
\partial v &= -\frac{1}{4}qe^{inR}w
\end{cases}.
\end{equation}
Remember that we denote $R = (z-z_0)^2 + (\wbar{z}-\wbar{z_0})^2 = \psi^2 + \wbar{\psi}^2$. To solve this pair of equations we need the Cauchy operators.

\begin{definition}
The {\it Cauchy operators} $\Ca, \wbar{\Ca}$ map a given function $f$ to the functions
\begin{equation}
\Ca f (z) = \frac{1}{\pi} \int_\Omega \frac{f(\xi)}{z-\xi} \,dm(\xi), \quad \wbar{\Ca} f (z) = \frac{1}{\pi} \int_\Omega \frac{f(\xi)}{\wbar{z}-\wbar{\xi}}\,dm(\xi),
\end{equation}
defined on the points where the integrals converge. Also, the \emph{Beurling operators} $\Pi, \wbar{\Pi}$ are defined as the principal value integrals
\begin{equation}
\Pi f (z) = -\frac{1}{\pi} \pv \int_\Omega \frac{f(\xi)}{(z-\xi)^2} dm(\xi), \quad \wbar{\Pi} f (z) = -\frac{1}{\pi} \pv \int_\Omega \frac{f(\xi)}{(\wbar{z} - \wbar{\xi})^2} dm(\xi).
\end{equation}
\end{definition}

\begin{theorem}
\label{cauchyOPbasics}
The operators $\Ca, \wbar{\Ca}, \Pi, \wbar{\Pi}$ have the following properties (all differentiations are done weakly):

\begin{enumerate}
\item If $f \in L^1(\Omega)$ then $\partial \wbar{\Ca} f = \wbar{\partial} \Ca f = f$.

\item If $f \in C(\wbar{\Omega})$, $f_{|\partial\Omega} = 0$ and $\wbar{\partial}f\in L^p(\Omega)$ or $\partial f \in L^p(\Omega)$, $2<p\leq\infty$, then respectively $\Ca \wbar{\partial} f = f$ or $\wbar{\Ca}\partial f = f$.

\item If $f \in L^p(\Omega)$, $p > 1$, then $\partial \Ca f = \Pi f$ and $\wbar{\partial} \wbar{\Ca} f = \wbar{\Pi} f$.

\item The operators $\Ca, \wbar{\Ca}, \Pi$ and $\wbar{\Pi}$ are bounded in $L^p(\Omega)$, $1<p<\infty$. If $2<p<\infty$, $\alpha = 1-\frac{2}{p}$, then $\Ca, \wbar{\Ca}:L^p(\Omega) \to C^\alpha(\wbar{\Omega})$ are bounded.

\item If $2<p<\infty$, $\alpha = 1-\frac{2}{p}$ then
\begin{equation}
\begin{split}
\lVert \Ca f \rVert_p, \lVert \wbar{\Ca} f \rVert_p &\leq C_p \lVert f \rVert_p \, , \\
\lVert \Ca f \rVert_\alpha, \lVert \wbar{\Ca} f \rVert_\alpha &\leq C_\alpha \lVert f \rVert_p \, , \\
\lVert \Pi f \rVert_p, \lVert \wbar{\Pi} f \rVert_p &\leq B_p \lVert f \rVert_p \, ,
\end{split}
\end{equation}
for some $C_\alpha, C_p, B_p < \infty$.
\end{enumerate}
\end{theorem}

\begin{remark}
Note that the $C^\alpha$ and $L^p$ norms are taken only in $\wbar{\Omega}$ unless otherwise specified.
\end{remark}

\begin{proof}
The results follow from theorems 1.14, 1.19, the remark after 1.20, 
theorem 1.35
and chapter I, section 9 in \cite{vekua}.
\end{proof}

Using property 1 we can solve the pair of equations \eqref{schrodingerReduced}. It is easily seen that $\wbar{\partial} f = g$ is equivalent to $f = g_0 + \Ca g$ for some holomorphic $g_0$. If we wanted to find all the solutions of the system, we would need to express $g_0$ with the help of boundary values of $f$. This is not needed since we are only interested in finding some solutions. Thus we may choose $g_0$ as we wish.

For $\wbar{\partial} v = -\frac{1}{4}qe^{inR}w$ we take the solution $v = - \frac{1}{4} \Ca\big( e^{inR}qw\big)$ and for $\partial w = e^{-inR} v$ we choose the solution $w = 1 + \wbar{\Ca}\big( e^{-inR} v\big)$. By combining these two we get an integral equation for $w$ of the form
\begin{equation}
\label{fixedPointEqn}
w = 1 - \frac{1}{4}\wbar{\Ca}\big( e^{-inR} \Ca(e^{inR}qw)\big).
\end{equation}
We will show that the operator on the right hand side is a contraction in a suitable H\"older space. To prove that, we need to have an estimate for the H\"older norm of $e^{inR}$. We will also need to integrate by parts at some point to take advantage of the oscillations of $e^{inR}$. Note that the operator is composition of operators of the form $\Ca(e^{\pm inR} \cdot)$. We will prove an integration by parts formula for those.

\begin{lemma}
\label{expHolder}
If $n > 1$, $z_0\in\Omega$ and $\alpha \in ]0, 1[$ then 
\begin{equation}
\lVert e^{i n R} \rVert_\alpha \leq 11 n^\alpha.
\end{equation}
\end{lemma}
\begin{proof}
For positive $r$ the inequality $r \leq c r^\alpha$ holds if and only if $r \leq c^{\frac{1}{1-\alpha}}$, so $|z_1 - z_2| \leq n^{\alpha-1} |z_1 - z_2|^\alpha$ when $|z_1 - z_2| < n^{-1}$. By a two dimensional mean value inequality we get
\begin{equation}
\begin{split}
\sup_{z_1, z_2 \in \Omega}& \frac{\big| e^{inR(z_1)} - e^{inR(z_2)}\big|}{|z_1 - z_2|^\alpha} \\
&\leq \sup_{\substack{z_1, z_2 \in \Omega\\ |z_1 - z_2| < n^{-1}}} \frac{\big| e^{inR(z_1)} - e^{inR(z_2)}\big|}{|z_1 - z_2|^\alpha} + \sup_{\substack{z_1, z_2 \in \Omega \\ |z_1 - z_2| \geq n^{-1}}} \frac{\big| e^{inR(z_1)} - e^{inR(z_2)}\big|}{|z_1 - z_2|^\alpha} \\
&\leq \sup_{\substack{z_1, z_2 \in \Omega\\ |z_1 - z_2| < n^{-1}}} \lVert 4 i n (z - z_0) e^{inR(z)} \rVert_{L^\infty(\Omega)} \frac{|z_1 - z_2|^{\phantom{\alpha}}}{|z_1 - z_2|^\alpha} + \sup_{\substack{z_1, z_2 \in \Omega \\ |z_1 - z_2| \geq n^{-1}}} \frac{2}{|z_1 - z_2|^\alpha} \\
&\leq 8n n^{\alpha - 1} + 2n^\alpha = 10 n^\alpha.
\end{split}
\end{equation}
Note that $R(z) \in \R$. Thus
\begin{equation}
\lVert e^{i n R(z)} \rVert_\alpha \leq \lVert e^{i n R(z)} \rVert_\infty + 10 n^\alpha = 1 + 10 n^\alpha \leq 11 n^\alpha.
\end{equation}
\end{proof}

\begin{lemma}
\label{intByParts}
If $n > 0$, $g \in C(\wbar{\Omega})$, $g_{|\partial\Omega} = 0$, $z_0 \in \Omega \setminus \supp g$ and $\wbar{\partial} g \in L^p(\Omega)$ or $\partial g \in L^p(\Omega)$, $2<p\leq\infty$, then respectively
\begin{equation}
\begin{split}
\Ca\big( e^{\pm inR} g \big) &= \frac{\pm 1}{2in} \Big( e^{\pm inR} \frac{g}{\wbar{z}-\wbar{z_0}} - \Ca\big( e^{\pm inR} \wbar{\partial} \frac{g}{\wbar{z}-\wbar{z_0}} \big) \Big) \quad \text{or} \\
\wbar{\Ca}\big( e^{\pm inR} g \big) &= \frac{\pm 1}{2in} \Big( e^{\pm inR} \frac{g}{z-z_0} - \wbar{\Ca}\big( e^{\pm inR} \partial \frac{g}{z-z_0} \big) \Big).
\end{split}
\end{equation}
\end{lemma}
\begin{proof}
Let's prove the case $\wbar{\partial}g \in L^p(\Omega)$. Because $z_0 \notin \supp g$ we may differentiate weakly to get
\begin{equation}
\wbar{\partial} \big( e^{\pm inR} \frac{g}{\wbar{z}-\wbar{z_0}} \big) = \pm 2in e^{\pm inR} g + e^{\pm inR} \wbar{\partial} \frac{g}{\wbar{z}-\wbar{z_0}}.
\end{equation}
Because $z_0 \notin \supp g$, $g\in C(\wbar{\Omega})$ and $\wbar{\partial}g \in L^p(\Omega)$ the last two terms are in $L^p(\Omega)$, and thus also the first one. We may use the Cauchy operator. Note that $e^{\pm inR}g/(\wbar{z}-\wbar{z_0})$ is continuous and vanishes at the boundary. We get the formula
\begin{equation}
\Ca\big( e^{\pm inR} g \big) = \frac{\pm 1}{2in} \Big( e^{\pm inR} \frac{g}{\wbar{z}-\wbar{z_0}} - \Ca\big( e^{\pm inR} \wbar{\partial} \frac{g}{\wbar{z}-\wbar{z_0}} \big) \Big)
\end{equation}
by property 2 in \ref{cauchyOPbasics}. The case $\partial g \in L^p(\Omega)$ can be proven similarly.
\end{proof}

Our next task is to prove that the integral operator of the fixed point equation \eqref{fixedPointEqn} is a contraction in $C^\alpha(\wbar{\Omega})$, where $\alpha = 1 -\frac{2}{p}$. The next lemma is the non-trivial part of the proof. It is basically an integration by parts argument combined with the use of the cut-off function we constructed earlier.

\begin{lemma}
\label{lemma1}
Let $2<p<\infty$, $\alpha = 1- \frac{2}{p}$, $z_0\in\Omega$ and $n > 1$. If $g \in C^\alpha(\wbar{\Omega})$ and $\wbar{\partial} g \in L^p(\Omega)$ or $\partial g \in L^p(\Omega)$, then respectively
\begin{equation}
\begin{split}
\big\lVert \Ca( e^{-inR} g) \big\rVert_\alpha& \leq 400(C_\alpha+11) n^{-\frac{2}{p(2p+1)}} \big( \lVert g \rVert_\alpha + \lVert \wbar{\partial} g \rVert_p \big), \text{ or} \\
\big\lVert \wbar{\Ca}( e^{-inR} g) \big\rVert_\alpha &\leq 400(C_\alpha+11) n^{-\frac{2}{p(2p+1)}} \big( \lVert g \rVert_\alpha + \lVert \partial g \rVert_p \big).
\end{split}
\end{equation}
\end{lemma}
\begin{proof}
Let us prove the case $\wbar{\partial} g \in L^p(\Omega)$ first. Let $0 < \delta < 1$ and take $h$ as given by lemma \ref{testFexist}. Now $z_0 \notin \supp(hg)$, $hg \in C(\wbar{\Omega})$, it vanishes on $\partial \Omega$ and $\wbar{\partial}(hg) \in L^p(\Omega)$ so by lemma \ref{intByParts} we can integrate by parts to get
\begin{equation}
\begin{split}
\Ca&\big( e^{-inR} g\big) = \Ca\big(e^{inR} (1-h)g \big) + \Ca\big( e^{-inR} h g \big) = \Ca\big( e^{-inR}(1-h)g \big) \\
&- \frac{1}{2in}\Big( e^{-inR} \frac{h}{\wbar{z}-\wbar{z_0}}g - \Ca\big(e^{-inR}\wbar{\partial}\big(\frac{h}{\wbar{z}-\wbar{z_0}}\big)g\big) - \Ca\big(e^{-inR} \frac{h}{\wbar{z}-\wbar{z_0}} \wbar{\partial} g \big)\Big).
\end{split}
\end{equation}
\begin{tabular}{m{7.9cm}c}
Next we will use the inequalities on the right and the fact that $\lVert e^{-inR} \rVert_\alpha \leq 11n^\alpha$. The terms involving $h$ are estimated using lemma \ref{testFexist}. Note that $n^{-1} < n^{-2/p}$ and $\delta^{-1} < \delta^{-2}$ because $p>2$ and $\delta<1$, so
& 
\framebox[5.1cm][l]{$\begin{array}{l}
\lVert \Ca f \rVert_\alpha \leq C_\alpha \lVert f \rVert_p \\
\lVert f \rVert_{L^p(S)} \leq (m(S))^{1/p} \lVert f \rVert_\alpha \\
\lVert ab \rVert_\alpha \leq \lVert a \rVert_\alpha \lVert b \rVert_\alpha \\
\lVert a \rVert_\alpha \leq 2 \lVert a \rVert_{C^1(\wbar{\Omega})}
\end{array}$}
\end{tabular}
\begin{equation}
\begin{split}
\big\lVert \Ca&( e^{-inR} g) \big\rVert_\alpha \leq C_\alpha \lVert (1-h)g \rVert_p + \frac{1}{2n}\Big( \lVert e^{-inR} \rVert_\alpha \big\lVert \frac{h}{\wbar{z}-\wbar{z_0}} \big\rVert_\alpha \lVert g \rVert_\alpha \\
&\phantom{\leq} + C_\alpha \big\lVert \wbar{\partial} \frac{h}{\wbar{z}-\wbar{z_0}} \big\rVert_\infty \lVert g \rVert_p + C_\alpha \big\lVert \frac{h}{\wbar{z}-\wbar{z_0}} \big\rVert_\infty \lVert \wbar{\partial} g \rVert_p \Big) \\
&\leq C_\alpha (m(\mho))^{1/p} \lVert g \rVert_\alpha + \frac{1}{2n} \Big( 11 n^\alpha \cdot 2 \cdot \frac{200}{\delta^2} \lVert g \rVert_\alpha \\
&\phantom{\leq} + C_\alpha\frac{200}{\delta^2} \pi^{1/p} \lVert g \rVert_\alpha + C_\alpha \frac{2}{\delta} \lVert \wbar{\partial} g \rVert_p \Big) \\
&\leq C_\alpha(2\pi)^{1/p} \delta^{1/p} \lVert g \rVert_\alpha +200 (C_\alpha+11) \delta^{-2}n^{-2/p} \big( \lVert g \rVert_\alpha + \lVert \wbar{\partial} g \rVert_p \big) \\
&\leq 200(C_\alpha+11) (\delta^{1/p} + \delta^{-2}n^{-2/p}) \big( \lVert g \rVert_\alpha + \lVert \wbar{\partial} g \rVert_p \big).
\end{split}
\end{equation}
Because $n > 1$ we may choose $\delta = n^{-\frac{2}{2p+1}}$ and get
\begin{equation}
\big\lVert \Ca( e^{-inR} g) \big\rVert_\alpha \leq 400(C_\alpha+11) n^{-\frac{2}{p(2p+1)}} \big( \lVert g \rVert_\alpha + \lVert \wbar{\partial} g \rVert_p \big).
\end{equation}
The estimate for the case $\partial g \in L^p(\Omega)$ is deduced similarly.
\end{proof}

\begin{theorem}
\label{operatorNorm}
Let $2 < p < \infty$, $\alpha = 1-\frac{2}{p}$, $z_0\in\Omega$, $n > 1$ and $q_1, q_2 \in L^p(\Omega)$. Then for $f \in C^\alpha(\wbar{\Omega})$ we have
\begin{equation}
\big\lVert \Ca\big( e^{-inR} \wbar{\Ca}( e^{inR} q_1 f) \big) \big\rVert_\alpha, \big\lVert \wbar{\Ca}\big( e^{-inR} \Ca( e^{inR} q_2 f) \big) \big\rVert_\alpha \leq C n^{-\frac{2}{p(2p+1)}} \lVert f \rVert_\infty \, ,
\end{equation}
where $C = 400(C_\alpha+11)(C_\alpha+B_p)\max(\lVert q_1 \rVert_p, \lVert q_2 \rVert_p)$.
\end{theorem}
\begin{proof}
By the properties of the Cauchy and Beurling operators (theorem \ref{cauchyOPbasics}) we know that $\wbar{\Ca}(e^{inR}q_1f)$ and $\Ca(e^{inR}q_2f)$ are in $C^\alpha(\wbar{\Omega})$. Moreover $\wbar{\partial} \wbar{\Ca}(e^{inR}q_1f) = \wbar{\Pi}(e^{inR}q_1f) \in L^p(\Omega)$ and $\partial \Ca(e^{inR}q_2f) = \Pi(e^{inR}q_2f) \in L^p(\Omega)$. Thus we may use lemma \ref{lemma1} and get
\begin{equation}
\begin{split}
&\big\lVert \Ca\big( e^{-inR} \wbar{\Ca}( e^{inR} q_1 f) \big) \big\rVert_\alpha \leq 400(C_\alpha+11) n^{-\frac{2}{p(2p+1)}} \big( \big\lVert \wbar{\Ca}(e^{inR}q_1 f) \big\rVert_\alpha \\
&\phantom{\leq} + \big\lVert \wbar{\Pi}( e^{inR} q_1 f) \big\rVert_p \big) \leq 400(C_\alpha+11) n^{-\frac{2}{p(2p+1)}} (C_\alpha +B_p) \lVert q_1 \rVert_p \lVert f \rVert_\infty.
\end{split}
\end{equation}
The estimate for $\wbar{\Ca}\big( e^{-inR} \Ca(e^{inR} q_2 f) \big)$ is deduced similarly.
\end{proof}

\bigskip
Now we define the operator appearing in the right-hand side of the integral equation which defines the oscillating solutions to the Schr\"odinger equation.

\begin{definition}
\label{SSbarDefinitions}
Let $2<p<\infty$, $\alpha = 1-\frac{2}{p}$ and $q_1, q_2 \in L^p(\Omega)$ be the two potentials that we wish to prove equal. Let $z_0 \in \Omega$ and $n > \max\Big(1, \big(\frac{1}{2}C\big)^{\frac{p(2p+1)}{2}}\Big) = n_0$. We define the operators $S_{n,z_0}, \wbar{S}_{n,z_0}$ mapping $C^\alpha(\wbar{\Omega})$ to itself by
\begin{equation}
\begin{split}
S_{n,z_0} f  = -\frac{1}{4} \Ca \big( e^{-inR} \wbar{\Ca} (e^{inR} q_1 f) \big),\\
\wbar{S}_{n,z_0} f = -\frac{1}{4} \wbar{\Ca} \big( e^{-inR} \Ca ( e^{inR} q_2 f ) \big).
\end{split}
\end{equation}
\end{definition}
By theorem \ref{operatorNorm} and the inequality $\lVert f \rVert_\infty \leq \lVert f \rVert_\alpha$ we have the norm estimates $\lVert S_{n,z_0} \rVert_\alpha, \lVert \wbar{S}_{n,z_0} \rVert_\alpha \leq C n^{-\frac{2}{p(2p+1)}} < \frac{1}{2}$. The operators are thus contractions. It is also easily seen that the operators $g \mapsto \wbar{\Ca}\big( e^{inR}q_j g \big)$, $g \mapsto \Ca\big( e^{-inR}q_j g \big)$ from $L^\infty(\Omega)$ to $C^\alpha(\wbar{\Omega})$ are uniformly Lipschitz\footnote{We mean that the operator norm of the difference is bounded by the difference of their respective values of $z_0$} with respect to $z_0$. This implies the same property for $S_{n,z_0}$ and $\wbar{S}_{n,z_0}$.

\medskip
Our next goal is to solve the integral equations $f = 1 + S f$ and $g = 1 + \wbar{S}g$ and get some estimate for the remainder terms $Sf$ and $\wbar{S}g$. In the end we will also prove that using these $f$ and $g$ we get the right kind of solutions to the original equations $\Delta u + q_j u = 0$. The next lemma shows that although we looked for solutions in a H\"older space, the solutions are in fact smoother.

\begin{lemma}
\label{intEqn}
If $2<p<\infty$, $\alpha = 1-\frac{2}{p}$, $z_0 \in \Omega$ and $n > n_0$ the equations
\begin{equation}
\label{intEqnEqn}
\begin{split}
f = 1 + S_{n,z_0} f\\
g = 1 + \wbar{S}_{n,z_0} g
\end{split}
\end{equation}
have a unique solution in $C^\alpha(\wbar{\Omega})$ which depends continuously on $z_0$. Moreover these functions are in $W^{2,p}(\Omega)$.
\end{lemma}
\begin{proof}
The operators $1 + S_{n,z_0}$ and $1 + \wbar{S}_{n,z_0}$ are contractions in $C^\alpha(\wbar{\Omega})$ that depend continuously on $z_0$. By Banach's fixed point theorem \eqref{intEqnEqn} has a unique solution depending continuously on $z_0$.

Next we prove that $f \in W^{2,p}(\Omega)$. The proof for $g$ is similar. We have the equalities
\begin{equation}
\begin{split}
f &= 1 - \frac{1}{4}\Ca\big(e^{-inR}\wbar{\Ca}(e^{inR}q_1 f)\big),\\
\wbar{\partial} f &= -\frac{1}{4} e^{-inR}\wbar{\Ca}(e^{inR}q_1 f),\\
\partial f &= -\frac{1}{4} \Pi\big( e^{-inR} \wbar{\Ca}(e^{inR} q_1 f)\big),\\
\partial \wbar{\partial} f &= \frac{1}{2}in(z-z_0)e^{-inR}\wbar{\Ca}(e^{inR}q_1 f) - \frac{1}{4}q_1 f,\\
\wbar{\partial} \wbar{\partial} f &= \frac{1}{2}in(\wbar{z}-\wbar{z_0})e^{-inR}\wbar{\Ca}(e^{inR}q_1 f) - \frac{1}{4}e^{-inR}\wbar{\Pi}(e^{inR}q_1 f),\\
\partial\partial f &= \frac{1}{2} \Pi\big(in(z-z_0)e^{-inR}\wbar{\Ca}(e^{-inR}q_1 f)\big) - \frac{1}{4}\Pi(q_1 f),\\
\wbar{\partial}\partial f &= \frac{1}{2} \Pi\big( in (\wbar{z}-\wbar{z_0})e^{-inR}\wbar{\Ca}(e^{inR} q_1 f)\big) - \frac{1}{4}\Pi\big(e^{-inR}\wbar{\Pi}(e^{inR} q_1 f )\big).
\end{split}
\end{equation}
Note that $f \in L^\infty(\Omega)$ and so $q_1 f \in L^p(\Omega)$. The operators $\Ca,\wbar{\Ca},\Pi, \wbar{\Pi}$ are bounded in $L^p(\Omega)$, so $\lVert f \rVert_{W^{2,p}(\Omega)} \leq 1+c n \lVert f \rVert_{L^\infty(\Omega)} < \infty$ for some $c \in \R$.
\end{proof}

Next we will prove the most technical lemma. It tells us that the Sobolev $W^{1,p}$ norms of the remainder terms tend to zero when $n$ grows. The convergence is uniform in $z_0$. The proof contains in principle half of the ideas of the proof of lemma 3.4 in \cite{bukhgeim}. The other half is in theorem \ref{OPremainder}. The argumentation is basically the same but here we avoid some redundancy.

\begin{lemma}
\label{restNorms}
Let $2<p<\infty$, $\alpha = 1-\frac{2}{p}$ and $n > n_0$. For $z_0 \in \Omega$ let $f = 1 + S_{n,z_0} f$ and $g = 1 + \wbar{S}_{n,z_0} g$ be the functions given in lemma \ref{intEqn}. Then
\begin{equation}
\begin{array}{ll}
\displaystyle{\sup_{z_0\in\Omega}}\lVert S_{n,z_0} f \rVert_\alpha \longrightarrow 0, &\displaystyle{\sup_{z_0\in\Omega}}\lVert \wbar{S}_{n,z_0} g \rVert_\alpha \longrightarrow 0, \\
\displaystyle{\sup_{z_0\in\Omega}}\lVert \wbar{\partial} S_{n,z_0} f \rVert_\infty \longrightarrow 0, &\displaystyle{\sup_{z_0\in\Omega}}\lVert \partial \wbar{S}_{n,z_0} g \rVert_\infty \longrightarrow 0, \\
\displaystyle{\sup_{z_0\in\Omega}}\lVert \partial S_{n,z_0} f \rVert_p \longrightarrow 0, &\displaystyle{\sup_{z_0\in\Omega}}\lVert \wbar{\partial}\wbar{S}_{n,z_0} g \rVert_p \longrightarrow 0,
\end{array}
\end{equation}
when $n \longrightarrow \infty$.
\end{lemma}
\begin{proof}
We will prove the claim for $f$. The one for $g$ is proven identically. Note that $f \in C^\alpha(\wbar{\Omega})$. Because $n>n_0$ we have $\lVert S_{n,z_0} \rVert_\alpha \leq \frac{1}{2}$. Thus $\lVert f \rVert_\alpha \leq 1 + \lVert S_{n,z_0}f \rVert_\alpha \leq 1 + \frac{1}{2} \lVert f \rVert_\alpha$ so $\lVert f \rVert_\alpha \leq 2$ for every $n>n_0$. This implies
\begin{equation}
\lVert S_{n,z_0} f \rVert_\alpha \leq C n^{-\frac{2}{p(2p+1)}} \lVert f \rVert_\alpha \leq 2C n^{-\frac{2}{p(2p+1)}} \longrightarrow 0.
\end{equation}
Note that $C$ and $n_0$ do not depend on $z_0$.

The other limits are not so trivial. We will first prove that $\lVert \partial S_{n,z_0}f \rVert_p$ tends to zero. For each $z_0\in\Omega$ and $\epsilon > 0$ choose $q_{\epsilon,z_0} \in C^\infty_0(\Omega)$ such that $\lVert q_1 - q_{\epsilon,z_0} \rVert_p < \min\big( \frac{\epsilon}{2C_\alpha}, \frac{2\epsilon}{B_pC_p} \big)$, $z_0 \notin \supp q_{\epsilon,z_0}$ and it satisfies the bounds uniform in $z_0$ given explicitly in lemma \ref{lemmaB3}. By the properties of the Cauchy and Beurling operators (thm \ref{cauchyOPbasics}) we get
\begin{equation}
\partial S_{n,z_0} f = -\frac{1}{4} \Pi\big(e^{-inR} \wbar{\Ca} (e^{inR} (q_1-q_{\epsilon,z_0}) f )\big)  - \frac{1}{4} \Pi\big(e^{-inR} \wbar{\Ca} (e^{inR} q_{\epsilon,z_0} f )\big).
\end{equation}
Note that
\begin{equation}
\big\lVert \Pi \big( e^{-inR} \wbar{\Ca}(e^{inR}Qf)\big)\big\rVert_p \leq B_p \big\lVert \wbar{\Ca}(e^{inR} Qf) \big\rVert_p \leq 2 B_p C_p \lVert Q \rVert_p
\end{equation}
for every $Q \in L^p(\Omega)$. Using this, the formula for the derivative and integration by parts (thm \ref{intByParts}) we get
\begin{equation}
\begin{split}
\big\lVert \partial &S_{n,z_0} f \big\rVert_p \leq \frac{1}{2} B_p C_p \lVert q_1 - q_{\epsilon, z_0} \rVert_p + \frac{1}{4}B_p \big\lVert \wbar{\Ca}( e^{inR} q_{\epsilon,z_0} f ) \big\rVert_p\\
&\leq \epsilon + \frac{B_p}{8n} \Big\lVert \Big( e^{inR} \frac{q_{\epsilon,z_0}}{z-z_0} f - \wbar{\Ca}\big( e^{inR} \partial \frac{q_{\epsilon,z_0} f }{z-z_0}\big) \Big) \Big\rVert_p\\
&\leq \epsilon + \frac{B_p}{8n} \Big\lVert \frac{q_{\epsilon,z_0}}{z-z_0}\Big\rVert_p \lVert f \rVert_\infty + \frac{B_pC_p}{8n} \Big\lVert \partial \frac{q_{\epsilon,z_0}}{z-z_0} \Big\rVert_p \lVert f \rVert_\infty \\
&\phantom{\leq}+ \frac{B_pC_p}{8n} \Big\lVert \frac{q_{\epsilon,z_0}}{z-z_0} \Big\rVert_\infty \lVert \partial f \rVert_p.
\end{split}
\end{equation}
Next note that $\lVert \partial f \rVert_p = \frac{1}{4} \lVert \Pi\big(e^{-inR}\wbar{\Ca}(e^{inR}q_1 f)\big) \rVert_p \leq \frac{1}{4}B_p C_p \lVert q_1 \rVert_p \lVert f \rVert_\infty$. By $\lVert f \rVert_\infty \leq 2$ we get the estimate
\begin{equation}
\label{lemma5312est}
\begin{split}
\sup_{z_0\in\Omega} \big\lVert \partial S_{n,z_0}f \big\rVert_p &\leq \epsilon + \frac{B_p}{4n}\sup_{z_0\in\Omega} \Big\lVert \frac{q_{\epsilon,z_0}}{z-z_0} \Big\rVert_p + \frac{B_pC_p}{4n} \sup_{z_0\in\Omega} \Big\lVert \partial \frac{q_{\epsilon,z_0}}{z-z_0} \Big\rVert_p\\
&+ \frac{B_p^2C_p^2}{16n}\sup_{z_0\in\Omega}\Big\lVert \frac{q_{\epsilon,z_0}}{z-z_0} \Big\rVert_\infty \lVert q_1 \rVert_p \,.
\end{split}
\end{equation}
This tends to $\epsilon$ when $n$ grows because of the bounds given by lemma \ref{lemmaB3} for $q_{\epsilon,z_0}$. Since $\epsilon$ was arbitrary, the limit is zero.

The limit for $\wbar{\partial} S_{n,z_0} f$ follows quite similarly. By taking a smooth potential $q_{\epsilon,z_0}$ near $q_1$ we get the same estimate as in \eqref{lemma5312est} for $\lVert \wbar{\partial} S_{n,z_0}f \rVert_\infty$, except 
for different constants and that in the second term we have the $L^\infty$ norm instead of the $L^p$ norm.
\end{proof}

\bigskip

\begin{theorem}
\label{solExistence}
Assume that $q_1, q_2 \in L^p(\Omega)$, $2<p<\infty$, $n>n_0$ and $z_0 \in \Omega$. Then there are solutions $u_j \in W^{2,p}(\Omega)$ to the Schr\"odinger equations $\Delta u_j + q_j u_j = 0$ which have the form
\begin{equation}
\begin{split}
u_1 &= e^{in(z-z_0)^2}(1 + r_{n,z_0}),\\
u_2 &= e^{in(\wbar{z}-\wbar{z_0})^2}(1 + s_{n,z_0}),
\end{split}
\end{equation}
where
\begin{equation}
\label{solExistThmRemainder}
\lim_{n\to\infty} \sup_{z_0\in\Omega} \lVert r_{n,z_0} \rVert_{W^{1,p}(\Omega)} = \lim_{n\to\infty} \sup_{z_0\in\Omega} \lVert s_{n,z_0} \rVert_{W^{1,p}(\Omega)} = 0.
\end{equation}
\end{theorem}

\begin{proof}
Choose $f$ and $g$ as in lemma \ref{intEqn} and set $u_1 = e^{in(z-z_0)^2}f$, $u_2 = e^{in(\wbar{z}-\wbar{z_0})^2}g$. Now $r_{n,z_0} = S_{n,z_0} f$ and $s_{n,z_0} = \wbar{S}_{n,z_0}g$ satisfy the estimate \eqref{solExistThmRemainder} by lemma \ref{restNorms}.

It is a direct calculation to see that the functions $u_j$ satisfy the Schr\"odinger equation:
\begin{equation}
\begin{split}
\Delta u &= 4\partial\wbar{\partial} \big(e^{in(z-z_0)^2}f\big) = 4\partial\Big( e^{in(z-z_0)^2} \wbar{\partial}\big( 1 - \frac{1}{4}\Ca( e^{-inR} \wbar{\Ca}(e^{inR} q_1 f ))\big)\Big)\\
&= -\partial\big( e^{-in(\wbar{z}-\wbar{z_0})^2} \wbar{\Ca}(e^{inR}q_1 f)\big) = -e^{in(z-z_0)^2}q_1 f = -q_1 u,\\
\Delta v &= 4 \wbar{\partial}\partial \big(e^{in(\wbar{z}-\wbar{z_0})^2}g\big) = 4\wbar{\partial} \Big( e^{in(z-z_0)^2} \partial \big( 1 - \frac{1}{4}\wbar{\Ca}(e^{-inR}\Ca(e^{inR}q_2 g)) \big) \Big)\\
&= -\wbar{\partial}\big( e^{-in(z-z_0)^2}\Ca(e^{inR}q_2 g)\big) = -e^{in(\wbar{z}-\wbar{z_0})^2} q_2 g = -q_2 v.
\end{split}
\end{equation}
We know that $f,g \in W^{2,p}(\Omega)$ and $e^{in(z-z_0)^2}, e^{in(\wbar{z}-\wbar{z_0})^2} \in \C^\infty(\wbar{\Omega})$, so $u$ and $v$ are in $W^{2,p}(\Omega)$.
\end{proof}

\newpage
\section{The proof}
\label{sect_proof}
All the ingredients for the main theorem are ready. The proof is very short since all the nontrivial lemmas were proven in the previous sections. Basically we plug the special oscillating solutions constructed in section \ref{exist} into the orthogonality relation $0 = \int (q_1 - q_2) u_1 u_2$ proven in section \ref{orthogonality}. Then we use the stationary phase method which was proven in section \ref{approx} to get $q_1 = q_2$.

\begin{theoremSect}
\label{FINALtheorem}
Let $2 < p \leq \infty$, $q_1, q_2 \in W^{1,p}_\oplus(\Omega)$ and $C_{q_1} = C_{q_2}$. Then $q_1 = q_2$.
\end{theoremSect}
\begin{proof}
Because $\Omega$ is bounded we may assume that $p < \infty$. Take $n_0$ as in definition \ref{SSbarDefinitions}. Then for each $n>n_0$ and $z_0 \in \Omega$ take the solutions $u,v \in W^{2,p}(\Omega)$ of the Schr\"odinger equations as in theorem \ref{solExistence}. This is possible because $q_1, q_2 \in W^{1,p}_\oplus(\Omega) \subset L^p(\Omega)$. By the orthogonality relation of theorem \ref{ORTthm} we know that
\begin{equation}
0 = \frac{2n}{\pi}\int_\Omega (q_1 - q_2)uv dm = \int_\Omega \frac{2n}{\pi} e^{inR} (q_1 - q_2) \big(1 + r_{n,z_0} \big) \big(1 + s_{n,z_0} \big) dm,
\end{equation}
for every $n> n_0$ and $z_0 \in \Omega$.

Note that $r_{n,z_0} = f -1$ and $s_{n,z_0} = g -1$ are continuous with respect to $z_0$ by lemma \ref{intEqn}. Denote $\varepsilon_{n,z_0} = r_{n,z_0} + s_{n,z_0} + r_{n,z_0}s_{n,z_0}$. By the Sobolev embedding $W^{1,p}(\Omega) \to L^\infty(\Omega)$ we see that $\sup_{z_0\in\Omega} \lVert \varepsilon_{n,z_0} \rVert_{W^{1,p}(\Omega)} \longrightarrow 0$ when $n$ grows because the same was true for $r_{n,z_0}$ and $s_{n,z_0}$ by theorem \ref{solExistence}.

Next we use the stationary phase method theorems \ref{OPthm}, \ref{OPremainder} and the orthogonality relation proven in \ref{ORTthm} to get
\begin{equation}
\begin{split}
\lVert &q_1 - q_2 \rVert_2 = \lim_{n\to\infty}\Big\lVert \int_\Omega \frac{2n}{\pi}e^{inR}(q_1 - q_2) dm \Big\rVert_{L^2(\Omega,z_0)} \\
&= \lim_{n\to\infty} \Big\lVert - \int_\Omega \frac{2n}{\pi} e^{inR} (q_1 - q_2) \varepsilon_{n,z_0} dm \Big\rVert_{L^2(\Omega,z_0)} \\
&\leq \lim_{n\to\infty} P \lVert q_1 - q_2 \rVert_{W^{1,p}_\oplus(\Omega)} \sup_{z_0\in\Omega} \lVert \varepsilon_{n,z_0} \rVert_{W^{1,p}(\Omega)} = 0.
\end{split}
\end{equation}
Thus $q_1 = q_2$.
\end{proof}

\newpage
\section{A closer look at Bukhgeim's paper}
\label{mistakeAnalysis}
We discuss an unclear argument in \cite{bukhgeim}, namely is it enough to prove his formula \cite[(3.34)]{bukhgeim} to get the result for potentials in $L^p(\Omega)$? We could only prove that the boundary data determines the potential uniquely if $q$ is piecewise in $W^{1,p}$, $p>2$. We remark that before \cite{bukhgeim}, even for $C^\infty$ potentials, there were no global uniqueness results for the inverse problem of the 2D Schr\"odinger equation. After Bukhgeim's paper, recent results in 2D use his idea of having solutions whose phase functions are analytic with nondegenerate critical points. See for example \cite{imanuvilovUhlmannYamamoto} and \cite{guillarmouTzou}.

\begin{claim*}[Bukhgeim 3.4.]
The set
\begin{equation}
G = \{ u_1u_2 \mid \text{$u_j$ given by lemma \ref{solExistence}} \}
\end{equation}
is dense in $L^q(\Omega)$ for $1 \leq q < \infty$.
\end{claim*}

Next we will outline his proof. First he notes that
\begin{equation}
u_1u_2 = e^{inR}(1+r_{n,z_0} + s_{n,z_0} + r_{n,z_0}s_{n,z_0}),
\end{equation}
where $r_{n,z_0} = S_{n,z_0}f = \sum_{k=1}^\infty S_{n,z_0}^k 1$ and $s_{n,z_0} = \wbar{S}_{n,z_0}g = \sum_{k=1}^\infty \wbar{S}_{n,z_0}^k 1$.
Next he tells that because $C^1(\wbar{\Omega})$ is dense in $L^q$ it is enough to show that if $a \in C^1(\wbar{\Omega})$ and
\begin{equation}
\int_\Omega a u_1u_2 \,dm = 0
\end{equation}
for every $u_1u_2 \in G$ then $a = 0$. After that he proves that if $a$ is continuously differentiable, then
\begin{equation}
\lim_{n\to\infty} \int_\Omega \frac{2n}{\pi} a u_1u_2 \,dm = a(z_0).
\end{equation}
This he does by first showing that $\int_\Omega 2n e^{inR} a \,dm \longrightarrow \pi a(z_0)$. Then he integrates by parts to show that the remaining terms tend to zero when $n$ grows. The ideas of this last part have been the basis of lemma \ref{restNorms} and theorem \ref{OPremainder}.

It is unclear why it would be enough to assume that $a$ is continuously differentiable. He does prove that $C^1(\wbar{\Omega}) \cap G^{\perp} = \{0\}$. But this does not imply that $G$ is dense in $L^2$ unless some other properties of $G$ are used. For example if $\phi$ is some discontinuous $L^2(\Omega)$ function and
\begin{equation}
G = \Big\{ h \in L^2(\Omega) \,\Big|\, \int_\Omega h \phi \,dm = 0 \Big\} = \big( \operatorname{span} \{\phi\} \big)^\perp
\end{equation}
then $G^\perp = \big((\operatorname{span}\{\phi\})^\perp\big)^\perp = \operatorname{span}\{\phi\}$ because the latter is closed as a one dimensional subspace. Thus $C^1(\wbar{\Omega}) \cap G^\perp = \{ 0 \}$ but $\wbar{G} = G \neq L^2(\Omega)$.

\newpage
\section{Appendix}
\input{input_soboIntByParts.tex}

\input{input_technicalIntegrals.tex}

\input{input_lemmaB3.tex}

\end{document}

%% file: input_gaussianFourier.tex
\begin{lemma}
\label{gaussianFourier}
If $c > 0$ then the Fourier transform of $t \mapsto e^{-c t^2/2}$ is the mapping $\xi \mapsto \frac{1}{\sqrt{c}} e^{-\xi^2/(2c)}$.
\end{lemma}
\begin{proof}
This is a direct calculation using Cauchy's integral theorem. Let $c>0$ and $\xi \in \R$. Then
\begin{equation}
\begin{split}
\big( e^{-c \frac{t^2}{2}} \big)^{\wedge} (\xi) &= \frac{1}{\sqrt{2\pi}} \int_{-\infty}^\infty \!\!\! e^{-c \frac{t^2}{2} - i \xi t} dt = \frac{1}{\sqrt{2\pi}} e^{-\frac{\xi^2}{2c}} \int_{-\infty}^\infty \!\!\! e^{-\big( \sqrt{\frac{c}{2}} t + \frac{i \xi}{2\sqrt{c/2}} \big)^2} \! dt \\
&= \frac{1}{\pi c} e^{-\frac{\xi^2}{2c}} \int_{-\infty}^\infty \!\!\! e^{-\big( s + \frac{i \xi}{2\sqrt{c/2}} \big)^2} \! ds = \frac{1}{\sqrt{\pi c}} e^{-\frac{\xi^2}{2c}} \int_{-\infty}^\infty e^{-u^2} du \\
&= \frac{1}{\sqrt{c}} e^{-\frac{\xi^2}{2c}}
\end{split}
\end{equation}
by Cauchy's integral theorem. This is justified because the function given by $z \mapsto e^{-z^2}$ is analytic and for any $A \in \R$ we have
\begin{equation}
\Big\lvert \int_s^{s+i A} \!\!\! e^{-z^2} dz \Big\rvert \leq \int_s^{s+iA} \!\!\! \lvert e^{-z^2} \rvert d\sigma(z) = \int_s^{s+iA} \!\!\! e^{A^2 - s^2} d\sigma(z) = \lvert A \rvert e^{A^2 - s^2},
\end{equation}
which tends to zero when $s \longrightarrow \infty$ or $s \longrightarrow -\infty$ along the real line.
\end{proof}

%% file: input_OPremainder.tex
\begin{theorem}
\label{OPremainder}
Let $2<p\leq\infty$, $q \in W^{1,p}_\oplus(\Omega)$ and let $\{g_{z_0} \in W^{1,p}(\Omega) \mid z_0 \in \Omega\}$ be a family of functions measurable in $(z,z_0)$. Then
\begin{multline}
\Big\lVert \int_\Omega \frac{2n}{\pi} e^{in\big((z-z_0)^2 + (\wbar{z}-\wbar{z_0})^2\big)} q(z) g_{z_0}(z) dm(z) \Big\rVert_{L^2(\Omega,z_0)} \\
\leq P \lVert q \rVert_{W^{1,p}_\oplus(\Omega)} \esssup_{z_0\in\Omega} \lVert g_{z_0} \rVert_{W^{1,p}(\Omega)},
\end{multline}
where $P$ is a polynomial function of the lengths of the boundaries $\partial\Omega_j$ and the Sobolev embedding\footnote{See \cite[thms 1.4.3.1, 1.4.4.1]{grisvard}.} constants in $\lVert f \rVert_{C^{1-2/p}(\wbar{\Omega_j})} \leq E_{p,j} \lVert f \rVert_{W^{1,p}(\Omega_j)}$. Here we denoted by $\Omega_j$ the Lipschitz pieces where $q$ is $W^{1,p}$ smooth.
\end{theorem}
\begin{proof}
Because $\Omega$ is bounded we may assume that $p<\infty$. Denote $R = (z-z_0)^2 + (\wbar{z}-\wbar{z_0})^2$ and $\alpha = 1 - \frac{2}{p} \in \left]0,1\right[$. Take a finite number of pairwise disjoint open sets $\Omega_j \subset \Omega$ with Lipschitz boundary such that $q_j = q_{|\Omega_j} \in W^{1,p}(\Omega_j)$ and $\wbar{\Omega} = \cup \wbar{\Omega_j}$. First we get
\begin{equation}
\begin{split}
\int_\Omega &\frac{2n}{\pi} e^{inR} q(z) g_{z_0}(z) dm(z) = \int_\Omega \frac{2n}{\pi} e^{inR} q(z)\big(g_{z_0}(z) - g_{z_0}(z_0)\big) dm(z) \\
&+ g_{z_0}(z_0) \int_\Omega \frac{2n}{\pi} e^{inR} q(z) dm(z).
\end{split}
\end{equation}
The second term is easy. By theorem \ref{OPdefThm} we get
\begin{equation}
\begin{split}
\Big\lVert &g_{z_0}(z_0) \int_\Omega \frac{2n}{\pi} e^{inR} q(z) dm(z) \Big\rVert_2 \leq \esssup_{z_0\in\Omega} \lVert g_{z_0} \rVert_\infty \Big \lVert \int_\Omega \frac{2n}{\pi} e^{inR} q dm \Big\rVert_2 \\
&\leq \esssup_{z_0\in\Omega} \lVert g_{z_0} \rVert_\alpha \lVert q \rVert_2 \leq \pi^{\frac{1}{2}-\frac{1}{p}} \lVert q \rVert_p \esssup_{z_0\in\Omega} \lVert g_{z_0} \rVert_\alpha
\end{split}
\end{equation}
We wish to integrate the first term by parts. Note that $2ine^{inR} = (z-z_0)^{-1}\partial e^{inR}$ and $e^{inR} \in W^{1,\infty}(\Omega)$. To be able to integrate by parts (lemma \ref{soboIntByParts}) we need to show that $q_j(z) \frac{g_{z_0}(z) - g_{z_0}(z_0)}{z-z_0}$ is in $W^{1,1+\epsilon}(\Omega_j)$ for some $\epsilon >0$ given $z_0$. Denote $F = \frac{g_{z_0}(z) - g_{z_0}(z_0)}{z-z_0}$ and note that $g_{z_0} \in W^{1,p}(\Omega) \subset C^\alpha(\wbar{\Omega})$ by Sobolev embedding. Now in the distributional sense
\begin{equation}
\begin{split}
F &= \frac{g_{z_0}(z) - g_{z_0}(z_0)}{z-z_0} \in L^a(\Omega) \text{ for $1\leq a < \frac{2}{1-\alpha} = p$}, \\
\partial F &= \frac{\partial g_{z_0}(z)}{z-z_0} - \frac{g_{z_0}(z) - g_{z_0}(z_0)}{(z-z_0)^2} \in L^a(\Omega) \text{ for $1\leq a < \frac{2}{2-\alpha} = \frac{2p}{p+2}$}, \\
\wbar{\partial} F &= \frac{\wbar{\partial} g_{z_0}(z)}{z-z_0} \in L^a(\Omega) \text{ for $1 \leq a < \big( \frac{1}{p} + \frac{1}{2}\big)^{-1} = \frac{2p}{p+2}$},
\end{split}
\end{equation}
so $F \in L^{a_1}(\Omega)$ for $1 \leq a_1 < p$, $\partial_k F \in L^{a_2}(\Omega)$ for $1\leq a_2 < \frac{2p}{2+p}$ and $q_j \in W^{1,p}(\Omega_j) \subset L^\infty(\Omega_j)$ by Sobolev embedding (see \cite[thms 1.4.3.1, 1.4.4.1]{grisvard}). Thus $q_j F \in L^{a_1}(\Omega_j)$ for $1 \leq a_1 < p$ and $\partial_k(gF) = (\partial_k g) F + g\partial_k F \in L^r(\Omega_j)$ for $r$ such that $\frac{1}{r} \geq \frac{1}{p} + \frac{1}{a_1}$ and $r \leq a_2$. Taking $r < \frac{p}{2}$ satisfies the first condition. Because $p>2$ it is possible to choose $1\leq r < \min(\frac{p}{2}, a_2)$. Thus $q_j F \in W^{1,r}(\Omega_j)$ for some $r>1$ and we may integrate by parts.

In the next estimate we split the area of integration into pieces where $q$ is $W^{1,p}$ smooth, integrate by parts in each piece, and then use Hölder inequalities to get $g_{z_0}$, $q_j$ and $\partial q_j$ out of the integral. The second to last term requires $p>2$: the function $\lvert z - z_0 \rvert^{(\alpha-1)p'}$ is integrable if and only if $p>2$.
\begin{equation}
\begin{split}
\Big\lVert& \int_\Omega \frac{2n}{\pi} e^{inR} q(z) \big( g_{z_0}(z) - g_{z_0}(z_0)\big) dm(z) \Big\lVert_{L^2(\Omega,z_0)} \\
&\leq \sum_{j=1}^m \Big\lVert \int_{\Omega_j} \frac{2n}{\pi} e^{inR} q_j(z) \big( g_{z_0}(z) - g_{z_0}(z_0)\big) dm(z) \Big\lVert_{L^2(\Omega,z_0)} \\
&= \sum_{j=1}^m\Big\lVert \frac{1}{\pi} \int_{\Omega_j} \big(\partial e^{inR}\big) q_j(z) \frac{g_{z_0}(z) - g_{z_0}(z_0)}{i(z-z_0)} dm(z) \Big\lVert_{L^2(\Omega,z_0)} \\
&\leq \sum_{j=1}^m \Big\lVert \frac{1}{\pi i} \int_{\partial\Omega_j} \frac{\wbar{z}}{2} e^{inR} \Tr_{\Omega_j} q_j(z) \frac{g_{z_0}(z) - g_{z_0}(z_0)}{z-z_0} d\sigma(z) \Big\lVert_{L^2(\Omega,z_0)} \\
&\phantom{\leq} + \sum_{j=1}^m \Big\lVert \frac{1}{\pi i} \int_{\Omega_j} e^{inR} q_j(z) \frac{g_{z_0}(z) - g_{z_0}(z_0)}{(z-z_0)^2} dm(z) \Big\lVert_{L^2(\Omega,z_0)} \\
&\phantom{\leq} + \sum_{j=1}^m \Big\lVert \frac{1}{\pi i} \int_{\Omega_j} e^{inR} \partial q_j(z) \frac{g_{z_0}(z) - g_{z_0}(z_0)}{z-z_0} dm(z) \Big\lVert_{L^2(\Omega,z_0)} \\
&\phantom{\leq} + \sum_{j=1}^m \Big\lVert \frac{1}{\pi i} \int_{\Omega_j} e^{inR} q_j(z) \frac{\partial g_{z_0}(z)}{z-z_0} dm(z) \Big\lVert_{L^2(\Omega,z_0)} \\
&\leq \sum_{j=1}^m \esssup_{z_0\in\Omega} \lVert g_{z_0} \rVert_\alpha \lVert q_j \rVert_\infty \Big\lVert \frac{1}{2\pi} \int_{\partial\Omega_j} |z-z_0|^{\alpha-1} d\sigma(z) \Big\lVert_{L^2(\Omega,z_0)} \\
&\phantom{\leq} + \sum_{j=1}^m \esssup_{z_0\in\Omega} \lVert g_{z_0} \rVert_\alpha \lVert q_j \rVert_\infty \Big\lVert \frac{1}{\pi} \int_{\Omega_j} |z-z_0|^{\alpha-2} dm(z) \Big\lVert_{L^2(\Omega,z_0)} \\
&\phantom{\leq} + \sum_{j=1}^m \esssup_{z_0\in\Omega} \lVert g_{z_0} \rVert_\alpha \lVert \partial q_j \rVert_p \Big\lVert \frac{1}{\pi} \Big( \int_{\Omega_j} |z-z_0|^{(\alpha-1)p'} dm(z) \Big)^{1/p'}\Big\lVert_{L^2(\Omega,z_0)} \\
&\phantom{\leq} + \sum_{j=1}^m \esssup_{z_0\in\Omega} \lVert \partial g_{z_0} \rVert_p \lVert q_j \rVert_\infty \Big\lVert \frac{1}{\pi} \Big( \int_{\Omega_j} |z-z_0|^{-p'} dm(z) \Big)^{1/p'} \Big\lVert_{L^2(\Omega,z_0)},
\end{split}
\end{equation}
where $p^{-1} + p'^{-1} = 1$. By lemma \ref{technicalIntegrals} and some simple arithmetic
\begin{equation}
\begin{split}
&\Big\lVert \frac{1}{2\pi} \int_{\partial \Omega_j} |z-z_0|^{\alpha-1} d\sigma(z) \Big\rVert_{L^2(\Omega,z_0)} \leq \frac{1}{2\pi} \sqrt{\frac{\pi}{\alpha}}\sigma(\partial\Omega_j), \\
&\Big\lVert \frac{1}{\pi} \int_{\Omega_j} |z-z_0|^{\alpha-2} dm(z) \Big\rVert_{L^2(\Omega,z_0)} \leq \frac{2\sqrt{\pi}}{\alpha}, \\
&\Big\lVert \frac{1}{\pi} \Big( \int_{\Omega_j} |z-z_0|^{(\alpha-1)p'} dm(z)\Big)^{1/p'} \Big\rVert_{L^2(\Omega,z_0)} \leq \sqrt{\pi}, \\
&\Big\lVert \frac{1}{\pi} \Big( \int_{\Omega_j} |z-z_0|^{-p'} dm(z) \Big)^{1/p'} \Big\rVert_{L^2(\Omega,z_0)} \leq \sqrt{\pi}\big(1+\frac{1}{\alpha}\big).
\end{split}
\end{equation}
To finish we use the inequalities $\lVert q_j \rVert_p, \lVert \partial q_j \rVert_p \leq \lVert q \rVert_{W^{1,p}_\oplus(\Omega)}$ and Sobolev embedding to get $\lVert q_j \rVert_\infty \leq \lVert q_j \rVert_{C^\alpha(\wbar{\Omega_j})} \leq E_{p,j} \lVert q_j \rVert_{W^{1,p}(\Omega_j)} \leq E_{p,j} \lVert q \rVert_{W^{1,p}_\oplus(\Omega)}$ and $\lVert g_{z_0} \rVert_\alpha, \lVert \partial g_{z_0} \rVert_p, \lVert \wbar{\partial} g_{z_0} \rVert_p \leq E_p \lVert g_{z_0} \rVert_{W^{1,p}(\Omega)}$.
\end{proof}

%% file: basetestfunct1image.tex
\beginpicture
\setcoordinatesystem units <1.75220cm,1.75220cm>
\put {\phantom{.}} at -3.45154 -0.46052
\put {\phantom{.}} at -3.45154 1.41135
\put {\phantom{.}} at 3.96769 -0.46052
\put {\phantom{.}} at 3.96769 1.41135
\setlinear
\setplotsymbol({.})
\setshadesymbol <z,z,z,z> ({\fiverm .})

\put {$g$} [b] at -3.00000 0.05000
\put {$\gamma$} [b] at -3.00000 1.05000

\putrule from -3.44412 0.00000 to 3.96027 0.00000
\putrule from -3.00000 -0.04875 to -3.00000 0.04875
\put {$-3$} [t] at -3.00000 -0.04875
\putrule from -2.00000 -0.04875 to -2.00000 0.04875
\put {$-2$} [t] at -2.00000 -0.04875
\putrule from -1.00000 -0.04875 to -1.00000 0.04875
\put {$-1$} [t] at -1.00000 -0.04875
\putrule from 0.00000 -0.04875 to 0.00000 0.04875
\put {$0$} [t] at 0.00000 -0.04875
\putrule from 1.00000 -0.04875 to 1.00000 0.04875
\put {$1$} [t] at 1.00000 -0.04875
\putrule from 2.00000 -0.04875 to 2.00000 0.04875
\put {$2$} [t] at 2.00000 -0.04875
\putrule from 3.00000 -0.04875 to 3.00000 0.04875
\put {$3$} [t] at 3.00000 -0.04875

\plot 
-3.45154 0.00000
0.07957 0.00050 
0.08282 0.00001
0.09256 0.00002
0.10229 0.00006
0.11203 0.00013
0.12177 0.00027
0.13150 0.00050
0.14124 0.00084
0.15098 0.00133
0.16071 0.00198
0.17045 0.00283
0.18018 0.00389
0.18992 0.00517
0.19966 0.00668
0.20939 0.00843
0.21913 0.01043
0.22887 0.01266
0.23860 0.01513
0.24834 0.01783
0.25808 0.02076
0.26781 0.02390
0.27755 0.02724
0.28729 0.03078
0.29702 0.03450
0.30676 0.03839
0.31650 0.04244
0.32623 0.04664
0.33597 0.05097
0.34571 0.05543
0.35544 0.06000
0.36518 0.06467
0.37492 0.06944
0.38465 0.07429
0.39439 0.07922
0.40412 0.08421
0.41386 0.08925
0.42360 0.09435
0.43333 0.09949
0.44307 0.10467
0.45281 0.10987
0.46254 0.11510
0.47228 0.12035
0.48202 0.12560
0.49175 0.13087
0.50149 0.13614
0.51123 0.14141
0.52096 0.14668
0.53070 0.15193
0.54044 0.15718
0.55017 0.16241
0.55991 0.16763
0.56965 0.17282
0.57938 0.17800
0.58912 0.18315
0.59886 0.18827
0.60859 0.19337
0.61833 0.19844
0.62806 0.20348
0.63780 0.20849
0.64754 0.21346
0.65727 0.21840
0.66701 0.22330
0.67675 0.22817
0.68648 0.23300
0.69622 0.23780
0.70596 0.24256
0.71569 0.24728
0.72543 0.25196
0.73517 0.25660
0.74490 0.26120
0.75464 0.26577
0.76438 0.27029
0.77411 0.27478
0.78385 0.27922
0.80008 0.28654
0.80981 0.29088
0.81955 0.29518
0.82929 0.29944
0.83902 0.30366
0.84876 0.30784
0.85850 0.31198
0.86823 0.31608
0.87797 0.32014
0.88771 0.32417
0.89744 0.32815
0.90718 0.33210
0.91692 0.33601
0.92665 0.33988
0.93639 0.34372
0.94612 0.34752
0.95586 0.35128
0.96560 0.35500
0.97533 0.35869
0.98507 0.36235
0.99481 0.36596
1.00454 0.36955
1.01428 0.37310
1.02402 0.37661
1.03375 0.38009
1.04349 0.38354
1.05323 0.38695
1.06296 0.39033
1.07270 0.39368
1.08244 0.39699
1.09217 0.40027
1.10191 0.40353
1.11165 0.40675
1.12138 0.40993
1.13112 0.41309
1.14086 0.41622
1.15059 0.41932
1.16033 0.42239
1.17006 0.42543
1.17980 0.42844
1.18954 0.43142
1.19927 0.43438
1.20901 0.43731
1.21226 0.43828
1.22199 0.44117
1.23173 0.44403
1.24147 0.44686
1.25120 0.44967
1.26094 0.45246
1.27068 0.45522
1.28041 0.45795
1.29015 0.46066
1.29988 0.46334
1.30962 0.46600
1.32585 0.47037
1.33559 0.47296
1.34532 0.47553
1.35506 0.47808
1.36479 0.48060
1.37453 0.48311
1.38427 0.48558
1.39400 0.48804
1.40374 0.49047
1.41348 0.49289
1.42321 0.49528
1.43295 0.49765
1.44918 0.50155
1.45891 0.50387
1.47514 0.50768
1.48488 0.50994
1.49462 0.51219
1.50435 0.51441
1.51409 0.51661
1.52382 0.51880
1.53356 0.52096
1.54330 0.52311
1.55303 0.52524
1.56277 0.52735
1.57251 0.52944
1.58224 0.53152
1.59198 0.53358
1.60172 0.53562
1.61145 0.53764
1.62119 0.53965
1.63093 0.54164
1.64066 0.54362
1.65040 0.54558
1.66014 0.54752
1.66987 0.54944
1.67961 0.55135
1.68935 0.55325
1.69908 0.55513
1.70882 0.55699
1.71856 0.55884
1.72829 0.56068
1.73803 0.56250
1.74776 0.56431
1.75750 0.56610
1.76724 0.56787
1.77697 0.56964
1.78671 0.57139
1.79645 0.57312
1.80618 0.57485
1.81592 0.57655
1.82566 0.57825
1.83539 0.57993
1.85162 0.58271
1.86136 0.58436
1.87109 0.58599
1.88083 0.58762
1.89057 0.58923
1.90030 0.59083
1.91004 0.59241
1.91978 0.59399
1.92951 0.59555
1.93925 0.59710
1.95548 0.59967
1.96521 0.60119
1.97495 0.60270
1.98469 0.60420
1.99442 0.60568
2.01065 0.60814
2.02039 0.60960
2.03012 0.61105
2.03986 0.61249
2.05609 0.61486
2.07232 0.61721
2.08205 0.61860
2.09179 0.61999
2.10153 0.62136
2.11126 0.62272
2.12100 0.62408
2.13073 0.62543
2.14047 0.62676
2.15021 0.62809
2.15994 0.62941
2.16968 0.63072
2.17942 0.63202
2.18915 0.63331
2.19889 0.63459
2.20863 0.63586
2.21836 0.63713
2.22810 0.63839
2.23784 0.63963
2.24757 0.64087
2.25731 0.64210
2.26705 0.64333
2.27678 0.64454
2.28652 0.64575
2.29626 0.64695
2.30599 0.64814
2.31573 0.64932
2.32547 0.65050
2.33520 0.65166
2.34494 0.65282
2.35467 0.65397
2.36441 0.65512
2.37415 0.65626
2.38388 0.65739
2.39362 0.65851
2.40336 0.65962
2.41958 0.66147
2.42932 0.66256
2.43906 0.66365
2.44230 0.66402
2.45204 0.66510
2.46178 0.66617
2.47151 0.66724
2.48125 0.66830
2.49099 0.66935
2.50072 0.67040
2.51046 0.67144
2.52020 0.67247
2.52993 0.67350
2.53967 0.67452
2.54941 0.67554
2.55914 0.67655
2.56888 0.67755
2.57861 0.67854
2.58835 0.67954
2.59809 0.68052
2.60133 0.68085
2.61107 0.68182
2.62081 0.68279
2.63054 0.68376
2.64028 0.68472
2.65002 0.68567
2.65975 0.68662
2.66949 0.68756
2.68572 0.68912
2.69545 0.69005
2.70519 0.69097
2.71493 0.69189
2.72466 0.69280
2.73440 0.69370
2.74414 0.69460
2.75387 0.69550
2.76361 0.69639
2.77335 0.69728
2.78308 0.69816
2.79282 0.69903
2.80255 0.69990
2.81229 0.70077
2.82203 0.70163
2.83176 0.70248
2.84150 0.70333
2.85124 0.70418
2.86097 0.70502
2.87071 0.70585
2.88045 0.70669
2.89018 0.70751
2.89992 0.70834
2.90966 0.70915
2.91939 0.70997
2.92913 0.71078
2.93887 0.71158
2.94860 0.71238
2.95834 0.71318
2.96808 0.71397
2.97781 0.71475
2.98755 0.71554
2.99729 0.71631
3.00702 0.71709
3.01676 0.71786
3.02649 0.71863
3.03623 0.71939
3.04597 0.72014
3.05570 0.72090
3.06544 0.72165
3.07518 0.72239
3.08491 0.72314
3.09465 0.72387
3.10439 0.72461
3.11412 0.72534
3.12386 0.72606
3.13360 0.72679
3.14333 0.72751
3.15307 0.72822
3.16281 0.72893
3.17254 0.72964
3.18228 0.73034
3.19202 0.73104
3.20175 0.73174
3.21149 0.73243
3.22123 0.73312
3.23096 0.73381
3.24070 0.73449
3.25043 0.73517
3.26017 0.73585
3.26991 0.73652
3.27964 0.73719
3.28938 0.73785
3.29912 0.73852
3.30885 0.73918
3.31859 0.73983
3.32833 0.74048
3.33806 0.74113
3.34780 0.74178
3.35754 0.74242
3.36727 0.74306
3.37701 0.74370
3.38675 0.74433
3.39648 0.74496
3.40622 0.74559
3.41596 0.74621
3.42569 0.74683
3.43543 0.74745
3.44517 0.74807
3.45490 0.74868
3.46464 0.74929
3.47437 0.74990
3.48411 0.75050
3.49385 0.75110
3.50358 0.75170
3.51332 0.75229
3.52306 0.75288
3.53279 0.75347
3.54253 0.75406
3.55876 0.75503
3.56200 0.75522
3.57174 0.75580
3.58148 0.75638
3.59121 0.75695
3.60095 0.75752
3.61069 0.75809
3.62042 0.75865
3.63016 0.75922
3.63990 0.75977
3.64963 0.76033
3.65937 0.76089
3.66911 0.76144
3.67884 0.76199
3.68858 0.76253
3.69182 0.76272
3.70156 0.76326
3.71130 0.76380
3.72103 0.76434
3.73077 0.76488
3.74051 0.76541
3.75024 0.76594
3.75998 0.76647
3.76972 0.76700
3.77945 0.76752
3.78919 0.76804
3.79893 0.76856
3.80866 0.76908
3.81840 0.76960
3.82813 0.77011
3.83138 0.77028
3.84112 0.77079
3.85085 0.77130
3.86708 0.77214
3.88331 0.77297
3.89954 0.77380
3.91901 0.77479
3.92875 0.77528
3.93199 0.77544
3.94822 0.77625
3.95796 0.77674
3.96769 0.77722
/
\plot
-3.45154 1.00000
1.07270 1.00050 
1.07594 0.99999
1.08568 0.99997
1.09542 0.99992
1.10515 0.99977
1.11489 0.99949
1.12463 0.99897
1.13436 0.99814
1.14410 0.99689
1.15384 0.99512
1.16357 0.99274
1.17331 0.98965
1.18305 0.98578
1.19278 0.98108
1.20252 0.97549
1.21226 0.96899
1.22199 0.96155
1.23173 0.95319
1.24147 0.94391
1.25120 0.93373
1.26094 0.92267
1.27068 0.91079
1.28690 0.88926
1.29664 0.87538
1.30638 0.86084
1.31611 0.84569
1.32585 0.82999
1.33559 0.81378
1.34532 0.79712
1.35506 0.78004
1.36479 0.76260
1.37453 0.74483
1.38427 0.72676
1.39400 0.70845
1.40374 0.68991
1.41348 0.67118
1.42321 0.65229
1.43944 0.62050
1.44918 0.60128
1.45891 0.58198
1.46865 0.56261
1.47839 0.54320
1.48812 0.52375
1.49786 0.50428
1.50760 0.48481
1.51733 0.46535
1.52707 0.44591
1.53681 0.42652
1.54654 0.40719
1.55628 0.38793
1.56602 0.36877
1.57575 0.34972
1.58549 0.33081
1.59523 0.31207
1.60496 0.29351
1.61470 0.27517
1.62444 0.25708
1.63417 0.23927
1.64391 0.22179
1.65365 0.20467
1.66338 0.18796
1.67312 0.17171
1.68285 0.15595
1.69259 0.14074
1.70233 0.12614
1.71206 0.11218
1.72180 0.09893
1.73154 0.08644
1.74127 0.07474
1.75101 0.06388
1.76075 0.05390
1.77048 0.04483
1.78022 0.03667
1.78996 0.02945
1.79969 0.02316
1.80943 0.01778
1.81917 0.01327
1.82890 0.00958
1.83864 0.00666
1.84838 0.00442
1.85811 0.00278
1.86785 0.00163
1.87759 0.00088
1.88732 0.00043
1.89706 0.00018
1.90679 0.00007
1.91653 0.00002
1.92627 0.00000
3.96769 0.00050
/
\endpicture

%% file: testfunct3imageuusi.tex
\beginpicture
\setcoordinatesystem units <2.39234cm,2.39234cm>
\put {\phantom{.}} at -1.25400 -0.71769
\put {\phantom{.}} at -1.25400 1.41863
\put {\phantom{.}} at 1.25400 -0.71769
\put {\phantom{.}} at 1.25400 1.41863
\setlinear
\setshadesymbol <z,z,z,z> ({\fiverm .})
\circulararc 360 degrees from 0.02500 0.35000 center at 0.00000 0.35000
\circulararc 360 degrees from 0.42500 0.35000 center at 0.40000 0.35000
\put {$z_0$} at 0.50 0.29

\circulararc 360 degrees from 1.00000 0.35000 center at 0.00000 0.35000
\circulararc 360 degrees from 0.89500 0.35000 center at 0.00000 0.35000
\circulararc 360 degrees from 0.79000 0.35000 center at 0.00000 0.35000
\circulararc 360 degrees from 0.61000 0.35000 center at 0.40000 0.35000
\circulararc 360 degrees from 0.50500 0.35000 center at 0.40000 0.35000

\endpicture

%% file: testfunct2imageuusi.tex
\beginpicture
\setcoordinatesystem units <2.39234cm,2.39234cm>
\put {\phantom{.}} at -1.25400 -0.71769
\put {\phantom{.}} at -1.25400 1.41863
\put {\phantom{.}} at 1.25400 -0.71769
\put {\phantom{.}} at 1.25400 1.41863

\circulararc 360 degrees from 0.00000 0.02500 center at 0.00000 0.00000
\circulararc 360 degrees from 0.40000 0.02500 center at 0.40000 0.00000
\setlinear
\setplotsymbol({.})
\setshadesymbol <z,z,z,z> ({\fiverm .})
\putrule from -1.25150 0.00000 to 1.25150 0.00000
\plot 
-1.25331 0.00000
-0.88769 0.00050 
-0.88700 0.00001
-0.88630 0.00002
-0.88561 0.00004
-0.88491 0.00009
-0.88422 0.00018
-0.88353 0.00033
-0.88283 0.00055
-0.88214 0.00089
-0.88145 0.00136
-0.88075 0.00200
-0.88006 0.00284
-0.87936 0.00391
-0.87867 0.00524
-0.87798 0.00686
-0.87728 0.00881
-0.87659 0.01109
-0.87590 0.01374
-0.87520 0.01677
-0.87451 0.02021
-0.87381 0.02405
-0.87312 0.02831
-0.87243 0.03300
-0.87173 0.03812
-0.87104 0.04367
-0.87034 0.04965
-0.86965 0.05604
-0.86896 0.06285
-0.86826 0.07007
-0.86757 0.07768
-0.86688 0.08567
-0.86618 0.09404
-0.86549 0.10275
-0.86479 0.11181
-0.86410 0.12119
-0.86341 0.13088
-0.86271 0.14087
-0.86202 0.15113
-0.86133 0.16165
-0.86063 0.17242
-0.85994 0.18341
-0.85924 0.19462
-0.85855 0.20603
-0.85786 0.21763
-0.85716 0.22940
-0.85647 0.24132
-0.85578 0.25340
-0.85508 0.26561
-0.85439 0.27794
-0.85369 0.29038
-0.85300 0.30293
-0.85231 0.31557
-0.85161 0.32830
-0.85092 0.34110
-0.85023 0.35397
-0.84953 0.36690
-0.84884 0.37989
-0.84814 0.39292
-0.84745 0.40600
-0.84676 0.41910
-0.84606 0.43224
-0.84537 0.44541
-0.84468 0.45859
-0.84398 0.47179
-0.84329 0.48500
-0.84259 0.49821
-0.84190 0.51143
-0.84121 0.52464
-0.84051 0.53784
-0.83982 0.55103
-0.83912 0.56420
-0.83843 0.57734
-0.83774 0.59046
-0.83704 0.60355
-0.83635 0.61659
-0.83566 0.62959
-0.83496 0.64254
-0.83427 0.65542
-0.83357 0.66825
-0.83288 0.68099
-0.83219 0.69366
-0.83149 0.70623
-0.83080 0.71871
-0.83011 0.73107
-0.82941 0.74331
-0.82872 0.75542
-0.82802 0.76739
-0.82733 0.77920
-0.82664 0.79085
-0.82594 0.80231
-0.82525 0.81358
-0.82456 0.82463
-0.82386 0.83546
-0.82317 0.84605
-0.82247 0.85638
-0.82178 0.86644
-0.82109 0.87622
-0.82039 0.88568
-0.81970 0.89483
-0.81901 0.90364
-0.81831 0.91210
-0.81762 0.92019
-0.81692 0.92791
-0.81623 0.93523
-0.81554 0.94216
-0.81484 0.94866
-0.81415 0.95475
-0.81346 0.96042
-0.81276 0.96565
-0.81207 0.97046
-0.81137 0.97484
-0.81068 0.97880
-0.80999 0.98234
-0.80929 0.98548
-0.80860 0.98823
-0.80790 0.99061
-0.80721 0.99264
-0.80652 0.99435
-0.80582 0.99576
-0.80513 0.99689
-0.80444 0.99779
-0.80374 0.99848
-0.80305 0.99900
-0.80235 0.99937
-0.80166 0.99962
-0.80097 0.99979
-0.80027 0.99989
-0.79958 0.99995
-0.79889 0.99998
-0.79819 0.99999
-0.79750 1.00000
0.19738 1.00050 
0.19807 0.99999
0.19877 0.99998
0.19946 0.99995
0.20016 0.99990
0.20085 0.99981
0.20154 0.99966
0.20224 0.99942
0.20293 0.99907
0.20362 0.99858
0.20432 0.99792
0.20501 0.99706
0.20571 0.99597
0.20640 0.99461
0.20709 0.99296
0.20779 0.99098
0.20848 0.98866
0.20917 0.98597
0.20987 0.98290
0.21056 0.97943
0.21126 0.97554
0.21195 0.97124
0.21264 0.96650
0.21334 0.96134
0.21403 0.95575
0.21472 0.94973
0.21542 0.94329
0.21611 0.93644
0.21681 0.92919
0.21750 0.92154
0.21819 0.91351
0.21889 0.90511
0.21958 0.89636
0.22027 0.88727
0.22097 0.87785
0.22166 0.86813
0.22236 0.85812
0.22305 0.84783
0.22374 0.83728
0.22444 0.82649
0.22513 0.81548
0.22582 0.80424
0.22652 0.79282
0.22721 0.78120
0.22791 0.76942
0.22860 0.75747
0.22929 0.74539
0.22999 0.73317
0.23068 0.72082
0.23138 0.70837
0.23207 0.69581
0.23276 0.68316
0.23346 0.67042
0.23415 0.65762
0.23484 0.64474
0.23554 0.63180
0.23623 0.61881
0.23693 0.60577
0.23762 0.59270
0.23831 0.57958
0.23901 0.56644
0.23970 0.55327
0.24039 0.54009
0.24109 0.52689
0.24178 0.51368
0.24248 0.50047
0.24317 0.48725
0.24386 0.47404
0.24456 0.46084
0.24525 0.44766
0.24594 0.43449
0.24664 0.42134
0.24733 0.40823
0.24803 0.39515
0.24872 0.38211
0.24941 0.36911
0.25011 0.35617
0.25080 0.34329
0.25149 0.33048
0.25219 0.31774
0.25288 0.30508
0.25358 0.29252
0.25427 0.28005
0.25496 0.26770
0.25566 0.25547
0.25635 0.24337
0.25704 0.23142
0.25774 0.21962
0.25843 0.20800
0.25913 0.19655
0.25982 0.18531
0.26051 0.17428
0.26121 0.16347
0.26190 0.15290
0.26260 0.14260
0.26329 0.13257
0.26398 0.12282
0.26468 0.11339
0.26537 0.10427
0.26606 0.09550
0.26676 0.08707
0.26745 0.07902
0.26815 0.07134
0.26884 0.06406
0.26953 0.05718
0.27023 0.05071
0.27092 0.04466
0.27161 0.03904
0.27231 0.03385
0.27300 0.02908
0.27370 0.02475
0.27439 0.02083
0.27508 0.01733
0.27578 0.01423
0.27647 0.01152
0.27716 0.00917
0.27786 0.00717
0.27855 0.00550
0.27925 0.00412
0.27994 0.00301
0.28063 0.00213
0.28133 0.00146
0.28202 0.00096
0.28271 0.00060
0.28341 0.00036
0.28410 0.00020
0.28480 0.00010
0.28549 0.00005
0.28618 0.00002
0.28688 0.00001
0.28757 0.00000
0.51236 0.00050 
0.51305 0.00001
0.51374 0.00002
0.51444 0.00004
0.51513 0.00010
0.51582 0.00019
0.51652 0.00034
0.51721 0.00057
0.51791 0.00091
0.51860 0.00140
0.51929 0.00205
0.51999 0.00290
0.52068 0.00399
0.52137 0.00534
0.52207 0.00698
0.52276 0.00894
0.52346 0.01125
0.52415 0.01392
0.52484 0.01698
0.52554 0.02044
0.52623 0.02431
0.52692 0.02860
0.52762 0.03332
0.52831 0.03847
0.52901 0.04404
0.52970 0.05005
0.53039 0.05647
0.53109 0.06331
0.53178 0.07055
0.53247 0.07818
0.53317 0.08620
0.53386 0.09459
0.53456 0.10333
0.53525 0.11241
0.53594 0.12181
0.53664 0.13152
0.53733 0.14152
0.53803 0.15180
0.53872 0.16234
0.53941 0.17312
0.54011 0.18413
0.54080 0.19535
0.54149 0.20678
0.54219 0.21838
0.54288 0.23016
0.54358 0.24210
0.54427 0.25418
0.54496 0.26640
0.54566 0.27874
0.54635 0.29119
0.54704 0.30374
0.54774 0.31639
0.54843 0.32912
0.54913 0.34193
0.54982 0.35480
0.55051 0.36774
0.55121 0.38073
0.55190 0.39376
0.55259 0.40684
0.55329 0.41995
0.55398 0.43309
0.55468 0.44626
0.55537 0.45944
0.55606 0.47264
0.55676 0.48585
0.55745 0.49906
0.55814 0.51228
0.55884 0.52549
0.55953 0.53869
0.56023 0.55188
0.56092 0.56505
0.56161 0.57819
0.56231 0.59131
0.56300 0.60439
0.56369 0.61743
0.56439 0.63043
0.56508 0.64337
0.56578 0.65625
0.56647 0.66907
0.56716 0.68181
0.56786 0.69447
0.56855 0.70704
0.56925 0.71951
0.56994 0.73186
0.57063 0.74410
0.57133 0.75620
0.57202 0.76816
0.57271 0.77996
0.57341 0.79159
0.57410 0.80304
0.57480 0.81430
0.57549 0.82534
0.57618 0.83615
0.57688 0.84672
0.57757 0.85704
0.57826 0.86708
0.57896 0.87683
0.57965 0.88628
0.58035 0.89541
0.58104 0.90420
0.58173 0.91263
0.58243 0.92070
0.58312 0.92839
0.58381 0.93569
0.58451 0.94259
0.58520 0.94907
0.58590 0.95513
0.58659 0.96077
0.58728 0.96598
0.58798 0.97076
0.58867 0.97511
0.58936 0.97904
0.59006 0.98255
0.59075 0.98567
0.59145 0.98839
0.59214 0.99075
0.59283 0.99276
0.59353 0.99445
0.59422 0.99584
0.59491 0.99696
0.59561 0.99784
0.59630 0.99852
0.59700 0.99903
0.59769 0.99939
0.59838 0.99963
0.59908 0.99980
0.59977 0.99989
0.60047 0.99995
0.60116 0.99998
0.60185 0.99999
0.60255 1.00000
0.79750 1.00050 
0.79819 0.99999
0.79889 0.99998
0.79958 0.99995
0.80027 0.99989
0.80097 0.99979
0.80166 0.99962
0.80235 0.99937
0.80305 0.99900
0.80374 0.99848
0.80444 0.99779
0.80513 0.99689
0.80582 0.99576
0.80652 0.99435
0.80721 0.99264
0.80790 0.99061
0.80860 0.98823
0.80929 0.98548
0.80999 0.98234
0.81068 0.97880
0.81137 0.97484
0.81207 0.97046
0.81276 0.96565
0.81346 0.96042
0.81415 0.95475
0.81484 0.94866
0.81554 0.94216
0.81623 0.93523
0.81692 0.92791
0.81762 0.92019
0.81831 0.91210
0.81901 0.90364
0.81970 0.89483
0.82039 0.88568
0.82109 0.87622
0.82178 0.86644
0.82247 0.85638
0.82317 0.84605
0.82386 0.83546
0.82456 0.82463
0.82525 0.81358
0.82594 0.80231
0.82664 0.79085
0.82733 0.77920
0.82802 0.76739
0.82872 0.75542
0.82941 0.74331
0.83011 0.73107
0.83080 0.71871
0.83149 0.70623
0.83219 0.69366
0.83288 0.68099
0.83357 0.66825
0.83427 0.65542
0.83496 0.64254
0.83566 0.62959
0.83635 0.61659
0.83704 0.60355
0.83774 0.59046
0.83843 0.57734
0.83912 0.56420
0.83982 0.55103
0.84051 0.53784
0.84121 0.52464
0.84190 0.51143
0.84259 0.49821
0.84329 0.48500
0.84398 0.47179
0.84468 0.45859
0.84537 0.44541
0.84606 0.43224
0.84676 0.41910
0.84745 0.40600
0.84814 0.39292
0.84884 0.37989
0.84953 0.36690
0.85023 0.35397
0.85092 0.34110
0.85161 0.32830
0.85231 0.31557
0.85300 0.30293
0.85369 0.29038
0.85439 0.27794
0.85508 0.26561
0.85578 0.25340
0.85647 0.24132
0.85716 0.22940
0.85786 0.21763
0.85855 0.20603
0.85924 0.19462
0.85994 0.18341
0.86063 0.17242
0.86133 0.16165
0.86202 0.15113
0.86271 0.14087
0.86341 0.13088
0.86410 0.12119
0.86479 0.11181
0.86549 0.10275
0.86618 0.09404
0.86688 0.08567
0.86757 0.07768
0.86826 0.07007
0.86896 0.06285
0.86965 0.05604
0.87034 0.04965
0.87104 0.04367
0.87173 0.03812
0.87243 0.03300
0.87312 0.02831
0.87381 0.02405
0.87451 0.02021
0.87520 0.01677
0.87590 0.01374
0.87659 0.01109
0.87728 0.00881
0.87798 0.00686
0.87867 0.00524
0.87936 0.00391
0.88006 0.00284
0.88075 0.00200
0.88145 0.00136
0.88214 0.00089
0.88283 0.00055
0.88353 0.00033
0.88422 0.00018
0.88491 0.00009
0.88561 0.00004
0.88630 0.00002
0.88700 0.00001
0.88769 0.00000
1.25400 0.00050 
/
\putrule from -1.00000 0.10000 to -1.00000 -0.10000
\putrule from -0.89500 0.10000 to -0.89500 -0.10000
\putrule from -0.79000 0.10000 to -0.79000 -0.10000
\putrule from 0.19000 0.10000 to 0.19000 -0.10000
\putrule from 0.29500 0.10000 to 0.29500 -0.10000
\putrule from 0.50500 0.10000 to 0.50500 -0.10000
\putrule from 0.61000 0.10000 to 0.61000 -0.10000
\putrule from 0.79000 0.10000 to 0.79000 -0.10000
\putrule from 0.89500 0.10000 to 0.89500 -0.10000
\putrule from 1.00000 0.10000 to 1.00000 -0.10000
\endpicture

%% file: input_soboIntByParts.tex
Note that in the next theorems the domain is not necessarily the unit disc $\Omega$ but any bounded Lipschitz domain in $\R^2$. We use the definitions of \cite{grisvard} for Sobolev spaces on Lipschitz domains. The following is well defined according to \cite{gagliardo}.
\begin{definitionSect}
\label{lipschitz_trace}
Let $X\subset \R^2$ be a bounded Lipschitz domain and let $1<p<\infty$. Then the mapping $\Tr : Lip(\wbar{X}) \to L^p(\partial X)$, $\Tr u = u_{|\partial X}$ has a unique continuous extension from $W^{1,p}(X)$ to $L^p(X)$. It is called \emph{the trace operator of $W^{1,p}(X)$}.
\end{definitionSect}

\begin{theoremSect}[Intergration by parts]
\label{soboIntByParts}
Let $X \subset \R^2$ be a bounded Lipschitz domain.
Given $u \in W^{1,p}(X)$, $v \in W^{1,q}(X)$ with $1< p,q < \infty$, $p^{-1} + q^{-1} = 1$ and $k\in\{1,2\}$, we have the integration by parts formula
\begin{equation}
\int_X u \partial_k v \, dm = \int_{\partial X} n_k \Tr u \Tr v \, d\sigma - \int_X v \partial_k u \, dm,
\end{equation}
where $n_k$ is the component number $k$ of the outwards pointing unit normal vector of $\partial X$ defined almost everywhere and $\Tr$ is the trace operator defined in \ref{lipschitz_trace}.
\end{theoremSect}
\begin{proof}
According to \cite[thm 1.5.3.1]{grisvard} the proof is in \cite[chp 3, thm 1.1]{necas}. Note that Grisvard did not tell if the claim holds for $p,q \in \{1, \infty\}$.
\end{proof}

\begin{theoremSect}[Sobolev embedding]
\label{lipschitz_sobo_embedding}
Let $X\subset \R^2$ be a bounded Lipschitz domain, $2<p<\infty$ and $\alpha = 1 -\frac{2}{p}$. Then we have the continuous embedding $W^{1,p}(X) \subset C^\alpha(\wbar{X})$, 
\begin{equation}
\lVert u \rVert_{C^\alpha(\wbar{X})} \leq E_p \lVert u \rVert_{W^{1,p}(X)}.
\end{equation}
\end{theoremSect}
\begin{proof}
This is a consequence of theorems 1.4.4.1 and 1.4.3.1 in \cite{grisvard}.
\end{proof}

%% file: input_technicalIntegrals.tex
The next lemma is used only in the end of the proof of theorem \ref{OPremainder}.
\begin{lemmaSect}
\label{technicalIntegrals}
Let $X \subset \Omega$ be a Lipschitz domain, $z_0\in\Omega$ and $\beta > 0$. Then
\begin{equation}
\int_\Omega |z-z_0|^{-\beta} dm(z) \leq \int_\Omega |z|^{-\beta} dm(z) = \frac{2\pi}{2-\beta}, \qquad \text{if $\beta < 2$},
\end{equation}
and
\begin{equation}
\Big( \int_\Omega \big| \int_{\partial X} |z-z_0|^{-\beta} d\sigma(z) \big|^2 dm(z_0) \Big)^{1/2} \leq \sigma(\partial X) \sqrt{\frac{\pi}{1-\beta}}, \qquad \text{if $\beta < 1$}.
\end{equation}
\end{lemmaSect}
\begin{proof}
First note that $\int_\Omega |z-z_0|^{-\beta} dm(z) = \int_{B(z_0,1)}|z|^{-\beta}dm(z)$. Define $A = \Omega \cap B(z_0,1)$, $A_1 = \Omega \setminus B(z_0,1)$ and $A_2 = B(z_0,1) \setminus \Omega$. If $z\in A_1$ then $|z|<1$ so $|z|^{-\beta} > 1$. If $z\in A_2$ then $|z| \geq 1$ so $|z|^{-\beta} \leq 1$. Moreover $m(A_1) = m(A_2)$. Thus $\int_{A_2} |z|^{-\beta} dm(z) \leq \int_{A_1} |z|^{-\beta} dm(z)$ and so
\begin{multline}
\int_{B(z_0,1)} |z|^{-\beta} dm(z) = \int_A |z|^{-\beta} dm(z) + \int_{A_2} |z|^{-\beta} dm(z) \\\leq \int_A |z|^{-\beta} dm(z) + \int_{A_1} |z|^{-\beta} dm(z) = \int_\Omega |z|^{-\beta} dm(z).
\end{multline}
Also $\int_\Omega |z|^{-\beta}dm(z) = \int_{\partial\Omega}\int_0^1 r^{-\beta}\cdot r dr d\sigma(\theta) = \frac{2\pi}{2-\beta}$.

The second claim follows from Minkowski's integral inequality and the first claim:
\begin{equation}
\begin{split}
\Big( \int_\Omega \big| &\int_{\partial \Omega'} |z-z_0|^{-\beta} d\sigma(z) \big|^2 dm(z_0) \Big)^{1/2} \leq \int_{\partial\Omega'} \Big( \int_\Omega |z-z_0|^{-2\beta} dm(z_0) \Big)^{1/2} d\sigma(z) \\
&\leq \int_{\partial\Omega'} \sqrt{\frac{2\pi}{2-2\beta}} d\sigma(z) = \sigma(\partial\Omega') \sqrt{\frac{\pi}{1-\beta}},
\end{split}
\end{equation}
because $2\beta < 2$.
\end{proof}

%% file: input_lemmaB3.tex
This lemma is needed only in the proof of lemma \ref{restNorms}.
\begin{lemmaSect}
\label{lemmaB3}
Let $1\leq p<\infty$, $q\in L^p(\Omega)$, $z_0 \in \Omega$ and $\epsilon > 0$. Now there is a test function $q_{\epsilon, z_0} \in C_0^\infty(\Omega)$ satisfying $\lVert q - q_{\epsilon, z_0} \rVert_p < \epsilon$, $z_0 \notin \supp q_{\epsilon, z_0}$ and
\begin{equation}
\sup_{z_0\in\Omega} \Big\lVert \frac{q_{\epsilon,z_0}}{z-z_0} \Big\rVert_\infty, \sup_{z_0\in\Omega} \Big\lVert \partial \frac{q_{\epsilon,z_0}}{z-z_0} \Big\rVert_\infty, \sup_{z_0\in\Omega} \Big\lVert \wbar{\partial} \frac{q_{\epsilon,z_0}}{\wbar{z}-\wbar{z_0}} \Big\rVert_\infty < \infty.
\end{equation}
\end{lemmaSect}
\begin{proof}
Let $q_\epsilon \in C_0^\infty(\Omega)$ be such that $\lVert q - q_\epsilon \rVert_p < \frac{\epsilon}{2}$ and $q_\epsilon \neq 0$. Denote $\delta = \big(\frac{\epsilon}{2\pi^{1/p} \lVert q_\epsilon \rVert_\infty}\big)^\frac{p}{2}$. Let $\varphi \in C^\infty(\C)$ be such that $\varphi(z) = 0$ when $|z| < \frac{\delta}{2}$ and $\varphi(z) = 1$ when $|z| > \delta$. This can be constructed using the smooth step function of lemma \ref{perusSilea}.

Now choose $q_{\epsilon,z_0}(z) = q_\epsilon(z) \varphi(z-z_0)$. Then
\begin{equation}
\begin{split}
\lVert q - q_{\epsilon,z_0} \rVert_p &\leq \lVert q - q_\epsilon\rVert_p + \lVert q_\epsilon(1-\varphi) \rVert_p < \frac{\epsilon}{2} + \lVert q_\epsilon \rVert_\infty \lVert \chi_{\supp(1-\varphi)} \rVert_p \\
&\leq \frac{\epsilon}{2} + \lVert q_\epsilon \rVert_\infty \pi^{1/p} \delta^{2/p} = \epsilon.
\end{split}
\end{equation}
The three norm estimates are easy. It is enough to note that $q_{\epsilon,z_0}(z) = 0$ if $|z-z_0| < \frac{\delta}{2}$ and $\delta$ does not depend on $z_0$. The only thing $z_0$ does is to shift the factor $\varphi$. Thus we get uniform bounds for the norms involving $\varphi$ or its derivatives.
\end{proof}